\documentclass[mathpazo]{csam}

\RequirePackage{fix-cm}
\usepackage{mathrsfs,amsmath,amssymb,bm}
\usepackage{lipsum}
\usepackage{amsfonts}
\usepackage{epstopdf}
\usepackage{makecell, rotating}
\usepackage{multirow}
\usepackage{algorithm} 
\usepackage{algorithmic}
\usepackage{multirow}
\usepackage{subfigure}
\usepackage{color}

\newcommand{\bbR}{\mathbb{R}}
\newcommand{\bbC}{\mathbb{C}}
\newcommand{\bbZ}{\mathbb{Z}}

\newcommand{\by}{\bm{y}}
\newcommand{\bz}{\bm{z}}
\newcommand{\bx}{\bm{x}}
\newcommand{\tbx}{\tilde{\bm{x}}}

\newcommand{\bb}{\bm{b}}
\newcommand{\bn}{\bm{n}}
\newcommand{\br}{\bm{r}}
\newcommand{\bh}{\bm{h}}

\newcommand{\bN}{\bm{N}}

\newcommand{\bhphi}{\bm{\hat{\phi}}}
\newcommand{\bhpsi}{\bm{\hat{\psi}}}

\newcommand{\bu}{\bm{u}}

\newcommand{\bB}{\bm{B}}
\newcommand{\bA}{\bm{A}}
\newcommand{\bI}{\bm{I}}
\newcommand{\calS}{\mathcal{S}}
\newcommand{\calB}{\mathcal{B}}
\newcommand{\calP}{\mathcal{P}}

\newcommand{\calI}{\mathcal{I}}
\newcommand{\calD}{\mathcal{D}}

\newcommand{\hphi}{\hat{\phi}}
\newcommand{\hPhi}{\hat{\Phi}}
\newcommand{\hpsi}{\hat{\psi}}

\newcommand{\vzero}{\bm{0}}

\newcommand{\dom}{\mathrm{dom}}

\DeclareMathOperator{\intdom}{\mathrm{int}\dom}
\DeclareMathOperator{\ridom}{\mathrm{ri}\dom}
\DeclareMathOperator{\affdom}{\mathrm{aff}\dom}

\newcommand\tbbint{{-\mkern -16mu\int}}

\newcommand\dbbint{{-\mkern -19mu\int}}

\newcommand\bbint{
	{\mathchoice{\dbbint}{\tbbint}{\tbbint}{\tbbint}}
}

\allowdisplaybreaks[4]
\DeclareMathOperator*{\argmin}{\mathrm{argmin}}
\DeclareMathOperator*{\argmax}{\mathrm{argmax}}

\newcommand{\dist}{\mathrm{dist}}
\newtheorem{thm}{Theorem}
\newtheorem{assumption}[thm]{Assumption}

\theoremstyle{remark}

\begin{document}
	\title{An adaptive block Bregman proximal gradient method for computing stationary states of multicomponent phase-field crystal model}
	
	
	\author[C.L Bao et~al.]{Chenglong Bao\affil{1},
		Chang Chen\affil{1}~and Kai Jiang\affil{2}\corrauth}
	\address{\affilnum{1}\ Yau Mathematical Sciences Center, Tsinghua University, Beijing, 100084, China, \\
		\affilnum{2}\ School of Mathematics and Computational Science, 
		Hunan Key Laboratory for Computation and Simulation in Science and Engineering,
		Xiangtan University, Xiangtan, Hunan, 411105, China.}
	\emails{{\tt kaijiang@xtu.edu.cn} 
		}
	
	\begin{abstract}
In this paper, we compute the stationary states of the multicomponent phase-field crystal model by formulating it as a block constrained minimization problem. The original infinite-dimensional non-convex minimization problem is approximated by a finite-dimensional constrained non-convex minimization problem after an appropriate spatial discretization. To efficiently solve the above optimization problem, we propose a so-called adaptive block Bregman proximal gradient (AB-BPG) algorithm that fully exploits the problem's block structure. The proposed method updates each order parameter alternatively, and the update order of blocks can be chosen in a deterministic or random manner. Besides, we choose the step size by developing a practical linear search approach such that the generated sequence either keeps energy dissipation or has a controllable subsequence with energy dissipation. The convergence property of the proposed method is established without the requirement of global Lipschitz continuity of the derivative of the bulk energy part by using the Bregman divergence. The numerical results on computing stationary ordered structures in binary, ternary, and quinary component coupled-mode Swift-Hohenberg models have shown a significant acceleration over many existing methods.
	\end{abstract}
	
	\ams{54C40, 14E20.
	}
	\keywords{Multicomponent coupled-mode Swift-Hohenberg model, Stationary states, Adaptive block Bregman proximal gradient algorithm, Convergence analysis, Adaptive step size.}
	
	\maketitle

	\section{Introduction}

Multicomponent systems, such as alloys, soft matters, are an important class of materials, particularly for technical applications and processes. The microstructures of materials play a central role for a broad range of	industrial application, such as the mechanical property of the quality and the durability, optical device, high-capacity data storage devices\,\cite{garcke1999multiphase,nestler2005multicomponent,nestler2011phase,ofori2013multicomponent,taha2019phase}.
Advances in modeling and computation have significantly improved the understanding of the fundamental nature of microstructure and phase selection processes. Notable
contributions have been made through using the phase-field methodology\,\cite{elder2002modeling,elder2004modeling}, which has been
successful at examining mesoscale microstructure evolution over diffusive time scales. 
Recently, phase field crystal (PFC) models have been proposed to efficiently simulate
eutectic solidification, elastic anisotropy, solute drag, quasicrystal formation,
solute clustering and precipitation mechanisms\,\cite{chen2002phase,nestler2005multicomponent,steinbach2013phase}. Besides, binary and ternary component phase field models have attracted many research interests from the computation perspective\,\cite{badalassi2003computation,boyer2011numerical,yang2017numerical,alster2017simulating,alster2017phase,shen2018scalar,chen2020energy}.

The PFC model for a general class of multicomponent systems is formulated consisting
$s$ components in $d$ dimensional space.  
The concentrations of the components are described by $ s $ vector-valued functions
$\{\phi_i(\br)\}_{i=1}^{s}=(\phi_1(\br),\cdots,\phi_s(\br))$. The variable
$\phi_\alpha(\br)$, so-called order parameter, denotes the local fraction of phase $\alpha$.
The free energy functional of PFC model of a $s$-component system can be described by two
contributions, a bulk free energy $F[\{\phi_i(\br)\}_{i=1}^{s}]$ and an interaction potential 
$G[\{\phi_i(\br)\}_{i=1}^{s}]$, which drive the
density fields to become ordered by creating minimal in the free energy for these
states. Formally, we can write the free energy functional of the multicomponent system as 
\begin{equation}
E[\{\phi_i(\br)\}_{i=1}^s;\Theta]  = G[\{\phi_i(\br)\}_{i=1}^s;\Theta] +
F[\{\phi_i(\br)\}_{i=1}^s;\Theta],
\label{eq:model}
\end{equation}
where $ \Theta $ are relevant physical parameters. 
$F$ has polynomial or logarithmic formulation\,\cite{brazovskii1975phase,swift1977hydrodynamic,cross1993pattern} and $G$ is the 
interaction potential, such as high-order differential terms or convolution
terms\,\cite{brazovskii1975phase,savitz2018multiple}. 
Usually, some constraints are imposed on the PFC model, such as 
the mass conservation or incompressibility which means 
the order parameter $\{\phi_i(\br)\}_{i=1}^{s}$ belong to a feasible space.

To understand the fundamental nature of multicomponent systems, 
it often involves finding stationary states corresponding to ordered structures.
Denote $ V_i $ $(i = 1,2,\cdots,s) $ to be a feasible space of the $i$-th order parameter, 
the above problem is transformed into solving the minimization problem
\begin{equation}\label{eq:infitemin}
\min\quad E[\{\phi_i(\br)\}_{i=1}^s;\Theta], \text{~ s.t.~ } \phi_i(\br) \in V_i~(i = 1,2,\cdots, s),
\end{equation}
with different physical parameters $ \Theta $, which brings a tremendous
computational burden.

Different methods have been proposed for computing the stationary states
of multicomponent models and can be classified into two categories through different formulations and numerical techniques. 
One is to solve the steady nonlinear Euler-Lagrangian system of
\eqref{eq:infitemin} through different spatial discretization approaches.
The other class of approaches has been constructed via the formulation of the minimization problem \eqref{eq:infitemin}.
Among these methods, great efforts have been made for solving the nonlinear gradient flow equations.
Numerically, a gradient flow equation is discretized in both
space and time domains via different discretization techniques, and the stationary state is obtained with proper choices of initialization and step sizes. 
Typical energy stable schemes to gradient flows include convex
splitting\,\cite{xu2006stability,wise2009energy}, stabilized factor
methods\,\cite{shen2010numerical}, exponential time differencing
scheme\,\cite{zhu2016fast,du2019maximum,li2021stabilized,du2021maximum}, and recently developed
invariant energy quadrature\,\cite{yang2017numerical}, and scalar auxiliary variable approaches\,\cite{shen2018scalar}. 
When designing the fully discretized scheme, choosing proper time steps greatly impacts the performance of the methods. Most existing methods fix the time step or obtain the time step using heuristic methods\,\cite{qiao2011adaptive}.

In this work, instead of designing a numerical scheme for the gradient flow, we formulate the infinite-dimensional problem \eqref{eq:infitemin} as a finite-dimensional block-wise non-convex problem via appropriate spatial discretization schemes. Similar ideas have shown
success in computing stationary states of many physical problems,
such as the Bose-Einstein condensate\,\cite{wu2017regularized}, the calculation of density functional theory\,\cite{liu2015analysis} and one component PFC models\,\cite{jiang2020efficient}. It is noted that extending single component systems to multicomponent models is not trivial due to the following two reasons. First, the direct extension may fail convergence, even for the simplest steep descent method\,\cite{powell1973search}. Second, the update order's choice is not unique, leading to difficulty in analyzing the convergence of the numerical algorithm.
Therefore, it deserves to develop specialized numerical algorithms for the multicomponent systems, which require carefully examine or extend the corresponding analysis in a single component system.

In this paper, based on the newly developed optimization techniques, we propose an
adaptive block Bregman proximal gradient (AB-BPG) method
for solving the discretized problem with multi-block structures. The proposed algorithm has the desired energy dissipation and mass conservation properties. Theoretically, we prove the convergence property of the proposed algorithm without the global Lipschitz constant requirement. It guarantees that the algorithm converges to a stationary state given any initial point. Our method updates each order parameter function with adaptive step sizes by the line search method in practice. Compared with our previous
work\,\cite{jiang2020efficient} for single-component PFC models, the main contributions of this paper include
\begin{itemize}
	\item We propose an AB-BPG algorithm for arbitrary multiple components PFC models by
	involving in the block structures. The sequence of blocks can be updated
	either deterministically cyclic or randomly shuffled for each iteration. Moreover, the convergence property without the global Lipschitz constant assumption of the derivative of the bulk energy $F$ is
	rigorously proved once each block is updated at least once in every $T$ iterations;
	\item Together with a practical line search strategy, the sequence $\{\Phi^k=\{\phi_i^k\}_{i=1}^s\}$ generated by the proposed method has the generalized energy dissipation property, i.e., \ one of the following properties holds:
	\begin{enumerate}
		\item The sequence has the energy dissipation property;
		\item There exist a subsequence $\{\Phi^{k_j}\}\subset\{\Phi^k\}$ and a constant $M\in\mathbb{N}$ such that $1\leq k_{j+1}-k_j\leq M+1$ and $E(\Phi^{k_{j+1}})-E(\Phi^{k_j})\geq 0$, $\forall j$;
	\end{enumerate}
	\item  Extensive numerical experiments on computing the stationary states in the binary, ternary, and quinary 
	coupled-mode Swift-Hohenberg model has shown the advantages of the proposed method in terms of computational efficiency.
\end{itemize}

The rest of this paper is organized as follows. Section~\ref{sec:model} presents a
concrete multicomponent PFC model, i.e., the coupled-mode Swift-Hohenberg (CMSH) model,
and a spatial discretization formulation based on the
projection method.  In section~\ref{sec:method}, we propose the adaptive block
Bregman proximal gradient (AB-BPG) methods for solving the constrained
non-convex multi-block problems with proved convergence. 
In section~\ref{sec:appCMSH}, we apply the proposed approaches to the CMSH model with two choices of Bregman divergence.  Numerical results are reported in
section~\ref{sec:results} to illustrate the efficiency and accuracy of our algorithms.
	
	\section{Problem formulation}\label{sec:model}
	
	
	There are several multicomponent PFC models to describe the phase behaviors of alloys and soft-matters\,\cite{elder2007phase,alster2017simulating,alster2017phase,ofori2013multicomponent,taha2019phase,nestler2005multicomponent,nestler2005multicomponent,greenwood2011modeling,mermin1985mean,jiang2016stability,jiang2020stability}.
	In this work, we consider the coupled-mode Swift-Hohenberg (CMSH) model of
	multicomponent systems, which extends the classical Swift-Hohenberg model from
	one length scale to multiple length scales\,\cite{swift1977hydrodynamic,jiang2020stability,jiang2016stability}. 
	The CMSH  model allows the study of the formation and relative stability of periodic
	crystals and quasicrystals.  Define the integral average 
	\begin{equation}
	\bbint=
	\begin{cases}
	\frac{1}{|\Omega|} \int_{\Omega}, & \mbox{ for periodic crystals,}\\
	\lim\limits_{R\rightarrow \infty}\frac{1}{|B_R|}\int_{B_R}, & \mbox{ for quasicystals,}
	\end{cases}
	\end{equation}
	where $\Omega$ is a bounded domain and $B_R$ is a ball centered at origin with radii $R$. The free energy of the CMSH model for $s$ component system is
	\begin{align}\label{eq:energy}
	\begin{split}
	E[\{\phi_j(\br)\}_{j=1}^s] &= 
	\bbint \Big\{\dfrac{1}{2}\sum_{j=1}^s[(\nabla^2 + q_j^2)\phi_j(\br)]^2 +
	\sum_{\calI_{s,n}}\tau_{i_1,i_2,\cdots,i_s}\prod_{j=1}^{s}\phi_j^{i_j}(\br)\Big\}d\br,
	\end{split}	
	\end{align}
	where $\phi_j$ is the $j$-th order parameter, $ q_j >0$ is the $ j $-th characteristic length scale, $\tau_{i_1,i_2,\cdots,i_s}$ is interaction intensity related to the physical
	conditions, and $\calI_{s,n}$ is the index set defined as
	\begin{align*}
	\calI_{s,n} := \Bigg\{(i_1,i_1,\cdots,i_s):i_j\in\mathbb{N}~(j=1,2,\cdots,s),~ 1\leq
	\sum_{j=1}^s i_j\leq n \Bigg\}.
	\end{align*}
	Moreover, to conserve the average density, each order parameter $\phi_j$ satisfies
	\begin{align}
	\bbint \phi_j(\br)\,d\br = 0.
	\end{align}
	Theoretically, the ordered patterns including periodic and quasiperiodic
	structures correspond to local minimizers of the free energy functional \eqref{eq:energy} with
	respect to order parameters $ \phi_j$ $(j=1,2,\cdots,s)$. Thus, denote
	\begin{equation}\label{defined_GandF}
	G_j[\phi_j] = \bbint\dfrac{1}{2} [(\nabla^2 + q_j^2)\phi_j(\br)]^2d\br,~~ F[\{\phi_j(\br)\}_{j=1}^s] =\bbint \sum_{\calI_{s,n}}\tau_{i_1,i_2,\cdots,i_s}\prod_{j=1}^{s}\phi_j^{i_j}(\br)d\br,
	\end{equation}
	we focus on solving the minimization:
	\begin{align}\label{infpro}
	\begin{split}
	\min\quad &E[\{\phi_j(\br)\}_{j=1}^s] = \sum_{j=1}^s G_j[\phi_j(\br)] +  F[\{\phi_j(\br)\}_{j=1}^s]\\
	\mathrm{s.t.}\quad  & 	
	\bbint \phi_j(\br)d\br = 0, \quad j = 1,2,\cdots,s.
	\end{split}	
	\end{align}
	Throughout this paper, we assume that $n\leq 4$ in \eqref{defined_GandF} which means that bulk energy $F$ is a $4^{th}$ order polynomial.
	
	In this work, we pay attention to the periodic and quasiperiodic crystals and use the projection method\,\cite{jiang2014numerical} to discretize the CMSH free energy functional. After discretization, the infinite-dimensional problem \eqref{infpro} can be formulated to a finite-dimensional minimization problem in the form of
	\begin{equation}\label{finitepro}
	\begin{split}
	\min_{\hPhi} ~  E(\hPhi) = \sum_{j= 1}^{s}G_j(\bhphi_j) + F(\hPhi),
	\quad\mathrm{s.t.}\quad  e_1^\top \bhphi_j = 0,\quad j = 1,2,\cdots,s.
	\end{split}
	\end{equation}
	where $ \bhphi_j \in\bbC^{N_j} $ is the truncated Fourier coefficients corresponding to $\phi_j$ and $ \hPhi = \{\bhphi_j\}_{j =1}^s\in\bbC^{\bN} $ with $\bN = \sum_{j=1}^s\bN_j$. $ G_j(\bhphi_j) = \dfrac{1}{2}\langle \bhphi_j, \calD_j \bhphi_j\rangle $ and $ \calD_j\in\bbC^{\bN_j\times \bN_j} $ is a diagonal matrix.
	$ F(\{\bhphi_j\}_{j=1}^s) $ are $n$-dimensional convolutions in the reciprocal
	space. A direct evaluation of the nonlinear term $
	F(\{\bhphi_j\}_{j=1}^s) $ is extremely expensive, however, $F(\{\bhphi_j\}_{j=1}^s) $ is a simple multiplication in the $ n$-dimensional physical space.  
	Thus, the pseudospectral method takes the advantage of this observation
	by evaluating $ G_j(\bhphi_j)  $ in the Fourier space and
	$F(\{\bhphi_j\}_{j=1}^s)$ in the physical space via the Fast Fourier Transformation. For the self-containess, we leave the concrete discretization to the Appendix A.
	
	Compared to the single-component case, the problem \eqref{finitepro} has the block structure in terms of the order parameters. Besides, the objective function is the summation of non-separable bulk energy and separable interaction energy that facilitates the numerical algorithm's design. In the next section, we aim at designing efficient optimization-based algorithms for
	solving the non-convex and multi-block problem \eqref{finitepro}.
	It is worth noting that the proposed AB-BPG method can also be applied to other spatial discretization methods.
	
	\section{The proposed method}\label{sec:method}
In this section, we consider the minimization problem in the form of
\begin{equation}\label{GeneralFormulation}
\begin{split}
\min_X\quad& E(X) = f(X) + \sum_{j= 1}^s g_j(\bx_j) \\
\mathrm{s.t.}\quad&~ \bx_j\in\calS_j, \quad j = 1,2,\cdots,s.
\end{split}
\end{equation}
where $ \bx_j\in\bbC^{N_j} $, $\calS_j$ is the feasible space of variable $\bx_j$,
$ X = \{\bx_j\}_{j=1}^s \in\bbC^{N} $ with $ N = \sum_{j=1}^sN_j $.
It is easy to know that the problem \eqref{finitepro} can be reduced to
\eqref{GeneralFormulation} by setting $f= F$, $g_j=G_j $ and $\calS_j =
\{\bhphi_j:e_1^\top \bhphi_j = 0\}$. 
Throughout this paper, we make the following assumptions on the objective function (see the concrete definition of notations in the next subsection).
\begin{assumption}~\label{assum1}
	\begin{enumerate}
		\item  $f: \bbC^N\to(-\infty,\infty]$ is proper and continuously differential on $ \bbC^N $ but may not be convex.
		\item  $g_j:\bbC^{N_j}\to(-\infty,\infty]$ is proper, lower semicontinuous and convex.
		\item  $ \calS_j\subseteq \dom g_j $ is a nonempty, closed and convex set. 
		\item $ E $ is bounded below, and level bounded.	
		\item  For all $ X \in\dom E$, $\pi_j(\calB(X))\subseteq \ridom g_j $ where $ \calB(X) $ is the closed ball that contains the sub-level set $ [E\leq E(X)] \cap\prod_{j=1}^s\calS_j$.	
	\end{enumerate}	
\end{assumption}

Before presenting our numerical algorithm, we first introduce some notations and useful definitions in the following analysis. Then, we give an abstract framework of the first
order method with proved convergence. 
Two kinds of concrete numerical algorithms to the CMSH model based on the
abstract formulation for solving \eqref{finitepro} will be presented in the
section \ref{sec:appCMSH}.

\subsection{Notations and definitions}
\label{subsec:notations}

We denote $ \prod_{j=1}^s\calS_j:=\{X=(\bx_1,\ldots,\bx_s):\bx_j\in\calS_j,\forall j=1,\ldots,s\} $ and
the projection operator is defined as $ \pi_j:\bbC^N\to \bbC^{N_j}, X\mapsto \bx_j$.
For a subset $ S\subseteq \bbC^N $, $ \pi_j(S):=\{\bx_j:\bx_j = \pi_j(X),~\forall
X\in S\} $. Let $ C^k(S) $ be the $ k $-th continuously differential functions on $ S
$. The domain of a  function $ f: \bbC^N \to \bbR$ is defined as $ \dom
f:=\{x:f(x)<+\infty\} $ and the relative interior of $ \dom f $ is defined as $\ridom
f:= \{x\in\dom f:\exists~r >0,  B(x,r)\cap \affdom f \subseteq \dom f \} $, where $
\affdom f $ is the smallest affine set that contains $ \dom f $ and $
B(x,r):=\{y:\|y-x\|\leq r\} $. $ f $ is proper if $ f > -\infty $ and $  \dom f\neq
\emptyset$.  For $ \alpha\in\bbR $, $  [f \leq \alpha]:= \{x : f(x) \leq \alpha\}$
is the $ \alpha $-(sub)level set of $ f $. We say that $ f $ is level  bounded
if $ [f\leq \alpha] $ is bounded for all $ \alpha\in \bbR$.  $ f $ is lower
semicontinuous if all level set of $ f $ is closed. The subgradient of $ f $
at $ x \in \dom g $ is defined as $ \partial f(x) = \{u : f(y) - f(x) -
\langle u, y- x\rangle \geq 0, \forall y \in \dom f\} $. For $ a\in\bbR $,
we denote $ [a]^+:=\max\{0, a\} $. 
Moreover, the next table summarizes the notations used in this work. 
\begin{table}[!htpb]
	\centering
	\caption{Summary of notations}
	\begin{tabular}{|c|c|}
		\hline
		Notation         & Definition                     \\ \hline
		$s$        & the total number of blocks                  \\ \hline
		$b_k$      & the update block selected at the $ k $-th iteration       \\ \hline
		$n_j^k$    & the number of updates to $ \bx_j $ within the first $ k $ iterations \\ \hline
		$\bx_j^k$  & the value of $ \bx_j $ after the $ k $-th iteration     \\ \hline
		$\tbx_j^n$ & the value of $ \bx_j $ after $ n$-th update       \\ \hline
		$\by^k$  & the value of extrapolation point $ \by $ at the  $k-$th iteration    \\ \hline			
		$ w_k $& the extrapolation  weight used at the $ k- $th iteration    \\ \hline
		$m_k$     & $m_k = \argmin_{j}\{E(X^j):[k-M]^+\leq j\leq k\}$ \\ \hline
		$X$  & $ X = (\bx_1,\bx_2,\cdots, \bx_s) $   \\ \hline
		$\bx_{\neq j}$  & the value of $ (\bx_1,\cdots,\bx_{j-1},\bx_{j+1},\cdots,\bx_s) $   \\ \hline
		$\nabla_j f(X)$ & the partial gradient of $ f(X) $ with respect to $ \bx_j $     \\ \hline
	\end{tabular}
	\label{tab:notation}
\end{table}

Throughout this paper, we assume that $h:\bbC^N\to (-\infty, +\infty]$ is a strongly convex function and give the following useful definitions. 
\begin{definition}[Bregman divergence~\cite{li2019provable}] The Bregman
	divergence with respect to  $h\in C^1(\intdom h)$ is defined as
	\begin{align}\label{eqn:bregmandiv}
	D_h(x, y) = h(x)-h(y) - \langle \nabla h(y), x-y\rangle,\quad\forall~ (x,y)\in\dom h \times \intdom h.
	\end{align}
\end{definition}
It is noted that $D_h(x,y)\geq0$ and $D_h(x,y)=0$ if and only if $x=y$ due to
the strongly convexity of $h$. Moreover, $ D(\cdot,y) $ is also strongly convex
for any fixed $ y\in\intdom h $. The Bregman divergence with $h = \|\cdot\|^2/2$ reduces to the Euclidean distance. Using the Bregman divergence, we can generalize the Lipschitz condition to the so-called relative smoothness as follows.
\begin{definition}[Relative smoothness \cite{bauschke2016descent}]
	A function $ f$ is called $R_f$-smooth relative to $ h  $ if there exists $ R_f >0 $ such that 
	$ R_fh(x) - f(x) $ is convex for all  $ x\in\dom h $.
\end{definition}
When $h = \|\cdot\|^2/2$ , the relative smoothness reduces to the Lipschitz
smoothness which is an essential assumption in the analysis of many scheme in
computing gradient flows (such as semi-implicit\,\cite{elliott1993global} or
stabilized semi-implicit\,\cite{shen2010numerical}). However, this assumption greatly limits its application range in practical computation. In this work, we overcome this difficulty from numerical optimization using this novel tool. 
To deal with the multi-block problem of form \eqref{GeneralFormulation}, we generalize the definition of relative smoothness to block-wise function as the definition 2.4 in\,\cite{ahookhosh2019multi}.
\begin{definition}[Block-wise relative smoothness]
	For a block-wise function $ f(X) $, we call $ f(X) $ is $ (R_f^1, R_f^2,\cdots,R_f^s) $-smooth relative to $ (h_1,h_1,\cdots,h_s) $ if for each $ j $-th block and fixed $ \bx_{\neq j} $,
	\begin{itemize}
		\item $ h_j:\bbC^{N_j}\to(-\infty, +\infty] $ is $ \gamma_i $-strongly convex and $\dom h_j \subseteq  \{\bu: f(\bu, \bx_{\neq i}) <\infty\}  $. 
		\item $ F_j(\bu) := f(\bu, \bx_{\neq j}) $ is $ R_f^j $-smooth  relative to $ h_j $ with respect to $ \bu $.
	\end{itemize}
\end{definition}
Now, we are ready to present the numerical algorithm in the next subsection.

\subsection{AB-BPG algorithm}
Our main idea is to develop a kind of block coordinate
descent methods which minimize $E$ cyclically over each of
$\bx_1,\bx_2,\cdots,\bx_s$ while fixing the remaining blocks at their last
updated values, i.e., the Gauss-Seidel fashion. 
Precisely, given feasible $ X^k,X^{k-1} \in \prod_{j=1}^s\calS_j$, our 
AB-BPG method can pick $ b_k\in\{1,2,\cdots, s\} $
deterministically or randomly, then $ X^k = (\bx_1^k,\bx_2^k,\cdots,\bx_s^k) $ is
updated as follows
\begin{align}
\begin{cases}
\bx_i^{k+1} = \bx_i^k, & \text{if } i\neq b_k\\
\bx_i^{k+1}  = \argmin\limits_{\bz\in \calS_i}\Big\{g_i(\bz) +  \langle \nabla_i f(\by^k, \bx_{\neq i}^k), \bz - \by^k\rangle + \dfrac{1}{\alpha_k}D_{h_i}(\bz, \by^k)\Big\} & \text{if } i = b_k
\end{cases}
\label{Bprox}
\end{align}
where $ \alpha_k>0 $ is the step size and  $ \by^k $ is the extrapolation
\begin{align}
\by^k = (1+w_k)\bx_i^k - w_k\bx_i^{\text{prev}}=(1+w_k)\tbx_i^{n_i^k} - w_k\tbx_i^{n_i^k-1}.
\end{align}
The extrapolation weight $ w_k\in[0,\bar{w}] $ for some $ \bar{w} >0$ and $
\bx_i^{\text{prev}} $ is the value of $ \bx_i $ before it is updated to $ \bx_i^k $.
The definitions of $ \tbx_i $ and $ n_i^k $ can be found in Table \ref{tab:notation}.

To ensure the convergence, we make mild assumptions for the distance generating function $h_i,~i=1,2,\ldots,s$ and the block update order~\cite{xu2017globally}.
\begin{assumption}\label{assum2}
	There exist stongly convex functions $h_j, j=1,2,\ldots,s$, and positive constants $R_f^j, j=1,2,\ldots,s$ such that $\calS_j\subseteq\intdom h_j, j=1,2,\ldots,s$, and $f$ is  $ (R_f^1, R_f^2,\cdots,R_f^s)$-smooth relative to $ (h_1,h_1,\cdots,h_s) $.
\end{assumption}
\begin{assumption}\label{assum3}
	With any $ T \geq s$ consecutive iterations, each block should be updated at least once, i.e., for any $k$, it has $\{1,2,\ldots,s\}\subseteq\{b_k,b_{k+1},\ldots,b_{k+T}\}$.
\end{assumption}

In the following context, we show some properties of the iterates~\eqref{Bprox}.
First, define the $ j$-th block Bregman proximal gradient mapping $ T_\alpha^j:\prod_{i=1}^s\calS_i\to \bbC^{N_j} $ as
\begin{equation}\label{BPGmap}
T_{\alpha}^j(X) : = \argmin_{\bz\in \calS_j}\Big\{g_j(\bz) +  \langle \nabla_j f(X), \bz - \bx_j\rangle + \dfrac{1}{\alpha}D_{h_j}(\bz, \bx_j)\Big\}.
\end{equation}
The next lemma shows the well-posedness of $ T_\alpha^j $.
\begin{lemma}\label{lem:welldefine}
	Suppose Assumption \ref{assum1} and Assumption \ref{assum2} hold. The map $ T_\alpha^j $ defined in \eqref{BPGmap} is nonempty and single-valued from $ \prod_{i=1}^s\calS_i $ to $ \calS_j $.
\end{lemma}
\begin{proof}
	\nonumber
	Since $ f\in C^1(\bbC^n) $ and $ \calS_j\subset \intdom h_j $,  $ \nabla_j f(X) $ and $  D_{h_j}(\cdot, \bx_j) $ is well-defined.  Let 
	\begin{equation*}
	\begin{split}
	\Gamma_j(\bz) :=~& \alpha g_j(\bz) +  \alpha\langle \nabla_j f(X), \bz - \bx_j\rangle + D_{h_j}(\bz, \bx_j),
	\end{split}
	\end{equation*}
	we know that $ \Gamma_j $ is strongly convex due to the convexity of $ D(\cdot, \bx_j)$ and $ g_j $. Thus,  $ \Gamma_j $ is  coercive (\cite{bauschke2011convex}, Corollary 11.17). According to the Corollary 3.23 in\,\cite{brezis2010functional}, $ \Gamma_j $ achieves its minimum on $ \calS_j $ and the strongly convexity implies the uniqueness of minimum.
\end{proof}

\begin{remark}
	In Assumption \ref{assum1},  $ \calS_j$ is convex for all $ j $. Thus, Lemma
	\ref{lem:welldefine} implies that the iteration $ \bx_i^{k+1} =
	T_{\alpha_k}^i(\by^k,\bx_{\neq i}^k) $ in \eqref{Bprox} is well-defined as
	long as  $ X^k, X^{k-1}\in\prod_{j=1}^s\calS_j $ is set in the initialization.
\end{remark}
The next lemma shows that the mapping $T_{\alpha}^j$ has the descent property.
\begin{lemma}[Sufficient decrease property]\label{lem:Suffdescent}
	Suppose Assumption \ref{assum1} and Assumption \ref{assum2} hold. Let $X = \{\bx_i\}_{i=1}^s\in\prod_{i=1}^s\calS_i$ and $\bx_j^+ = T_{\alpha}^j(X)$, then we have 
	\begin{align}
	E(X) - 	E(\bx_j^+, \bx_{\neq j})\geq \left(\dfrac{1}{\alpha} - R_f^j\right)\dfrac{\gamma_j}{2} \|\bx_j^+ - \bx_j\|^2 ,
	\label{eq:Suffdescent}
	\end{align}
	where $ \gamma_j$ is the strong convexity coefficient of $ h_j $.
\end{lemma}
\begin{proof}
	Due to the block-wise relative smoothness of $
	f(X)$, the function $F_j(\bu) := R_f^j h_j(\bu) - f(\bu, \bx_{\neq j})$ is convex for any fixed $\bx_{\neq j}$.
	Thus, for any $\bz \in\intdom h_j$, we have
	\begin{equation}\label{ieq;RelativeSmooth}
	f(\bu,\bx_{\neq j})  \leq  f(\bz,\bx_{\neq j}) + \langle \nabla_j f(\bz,\bx_{\neq j}), \bu- \bz\rangle + R_f^j D_{h_j}(\bu, \bz).
	\end{equation}
	Together with the definition of mapping $T_{\alpha}^j$, we know that
	\begin{align*}
	& E(X) =  f(X) + g_j(\bx_j) + \sum_{i\neq j} g_i(\bx_i)\\
	& =  f(X) + \Big[\langle \nabla_j f(X), \bz- \bx_j\rangle + \dfrac{1}{\alpha}D_{h_j}(\bz, \bx_j) + g_j(\bz)\Big]_{\bz = \bx_j} + \sum_{i\neq j} g_i(\bx_i)\\
	&\geq  f(X) + \langle \nabla_j f(X), \bx_j^+- \bx_j\rangle + \dfrac{1}{\alpha}D_{h_j}(\bx_j^+, \bx_j) + g_j(\bx_j^+) + \sum_{i\neq j} g_i(\bx_i)\\
	& \geq f(\bx_j^+,\bx_{\neq j}) - R_f^j D_{h_j}(\bx_j^+, \bx_j) +  \dfrac{1}{\alpha}D_{h_j}(\bx_j^+, \bx_j) + g_j(\bx_j^+) + \sum_{i\neq j} g_i(\bx_i)\\
	& = E(\bx_j^+,\bx_{\neq j}) + \left(\dfrac{1}{\alpha} - R_f^j\right)D_{h_j}(\bx_j^+, \bx_j)\geq E(\bx_j^+,\bx_{\neq j}) + \left(\dfrac{1}{\alpha} - R_f^j\right)\dfrac{\gamma_j}{2} \|\bx_j^+- \bx\|^2.
	\end{align*}
	The second inequality holds by setting $\bu = \bx_j^+$ and $\bz = \bx_j$ in \eqref{ieq;RelativeSmooth}, and the last inequality holds from the $ \gamma_i $-convexity of $ h_i $.
\end{proof}

\textbf{Restart technique.}
Let $i = b_k$, the {iteration \eqref{Bprox}}
can be written as $X^{k+1} = T_{\alpha_k}^i(\by^k, \bx_{\neq i}^k)$, then Lemma~\ref{lem:Suffdescent} implies
that $E(\by^k,\bx_{\neq i}^k) \geq E(X^{k+1})$ as long as $\alpha_k \in (0, 1/R_f^i]$. However, this descent property is not enough for investigating the iterates $\{E(X^{k})\}$ as the relationship between $E(\by^k,\bx_{\neq i}^k)$ and $E(X^k)$ is not clear. Thus, it may lead to the energy oscillation in the sequence $\{E(X^k)\}$. To overcome this weakness, we propose a restart technique by setting $(\by^k,\bx_{\neq i}^k)=X^k$ when the sharp oscillation is detected.
In concrete,  given $ \alpha_k>0$ and $ i= b_k $, define 
\begin{equation}\label{eq:updateBPG}
\bz^k =\Gamma_{\alpha_k}^i(\by^k,\bx_{\neq i}^k) =\argmin_{\bz\in\calS_i}\left\{g_i(\bz) +  \langle \nabla_i f(\by^k,\bx_{\neq i}^k), \bz - \by^k\rangle + \dfrac{1}{\alpha_k}D_{h_i}(\bz, \by^k)\right\}.
\end{equation}
Given some non-negative integer constant $ M $, we define
\begin{equation}\label{def:m_k}
m_k = \argmax_{[k-M]^+\leq j\leq k }E(X^j), 
\end{equation}
and set $ X^{k+1} = X^k$, $w_{k+1}=0 $ if the following inequality
\begin{align}\label{restartcri}
E(X^{m_k})-E(\bz^k, \bx_{\neq i}^k)\geq \sigma\|\bx_i^k-\bz^k\|^2
\end{align}
does not hold where $\sigma>0$ is a small constant. Otherwise, we obtain $ X^{k+1} $ via $ \bx_i^{k+1} = \bz^k$, $\bx_j^{k+1} = \bx_j^k~(j\neq i) $ and update $ w_{k+1} \in [0,\bar{w}]$.
\begin{remark}
	\label{rmk:lem2}
	When $M=0$, it guarantees that $\{E(X^k)\}$ is monotone decreasing. When $M>0$, the scheme has the generalized descent property, i.e.,\ the subsequence $\{E(X^{m_k})\}$ is decreasing (see Lemma~\ref{lemma:de}).
\end{remark}

\textbf{Step size estimation.}
Let $i=b_k$, Lemma~\ref{lem:Suffdescent} shows that $E(\by^k,\bx_{\neq i}^k) \geq E(X^{k+1})$ is ensured by step size $\alpha_k \in (0, 1/R_f^i]$
which may be too conservative. Thus, we propose 
a non-monotone backtracking line
search method\,\cite{lu2013randomized} for finding the appropriate step $\alpha_k$ which is initialized by the similar
idea of BB method\,\cite{barzilai1988two}, i.e.,  \begin{align}\label{stepestimation}
\alpha_k = \begin{cases}
\alpha_0, & w_k = 0,\\
\dfrac{\langle u_i^k, u_i^k\rangle}{\langle u_i^k, v_i^k \rangle}\text{ or } \dfrac{\langle v_i^k, u_i^k\rangle}{\langle v_i^k, v_i^k\rangle},& w_k\neq 0,
\end{cases}	
\end{align}
where $ u_i^k =  \by^k- \bx_i^k$ and $v_i^k = \nabla_i
f(\by^k, \bx_{\neq i}^k)-\nabla_i f(X^k)$. Let $ \eta\in (0, \sigma]  $ be a constant and $
\bz^k $ is obtained from \eqref{eq:updateBPG}, we adopt the step size $ \alpha_k\in[\alpha_{\min},\alpha_{\max}]$
whenever the following inequality holds
\begin{equation}\label{linesearch}
\max(E(\by^k, \bx_{\neq i}^k),E(X^{m_k}))-E(\bz^k, \bx_{\neq i}^k)\geq\eta\|\by^k-\bz^k\|^2.
\end{equation}
Let $ R := \max\limits_{j=1,2,\cdots s} \{R_j\} $ and $ \gamma =
\max\limits_{j=1,2,\cdots s} \{\gamma_j\} $. Lemma~\ref{lem:Suffdescent}, the inequality \eqref{linesearch} holds whenever $0<\alpha_{\min}<\gamma/(2\eta +\gamma R)$. Thus, the line search scheme will terminate in finite iterations. In summary, we present the detailed algorithm for estimating step sizes in Algorithm~\ref{alg:eststep} and the proposed AB-BPG method in Algorithm~\ref{alg:BadaAPG}.
\begin{algorithm}[H]
	\caption{Estimation of $\alpha_k$ at $\by^k$ with respect to block $ i $}
	\label{alg:eststep}
	\begin{algorithmic}[1]
		\STATE {\bf Inputs:} $X^k$, $\by^k$, $\varsigma\in(0,1)$ and $\eta>0$, $ \alpha_0,\alpha_{\min},\alpha_{\max} >0 $.
		\STATE Initialize $\alpha_k$ by \eqref{stepestimation};
		\STATE Calculate the smallest index $ \ell \geq 0$  such that  \eqref{linesearch} holds at $\bz^k$ defined in \eqref{eq:updateBPG} with step size  $ \varsigma^{\ell}\alpha_k \geq \alpha_{\min}$.
		\STATE {\bf Output:} step size $\alpha_k = \max(\min(\varsigma^{\ell}\alpha_k, \alpha_{\max}), \alpha_{\min})$.
	\end{algorithmic}
\end{algorithm}

\begin{algorithm}[H]
	\caption{AB-BPG method}
	\label{alg:BadaAPG}
	\begin{algorithmic}[1]
		\STATE Initialize $X^{-1}=X^0\in\prod_{j=1}^s\calS_j$, $\sigma\geq\eta>0$ and $ w_0 = 0 $, $ \bar{w},\alpha_0, M \geq 0 $, $ k = 0 $.
		\WHILE{stopping criterion is not satisfied}
		\STATE Pick $ i = b_k \in\{1,2,\cdots, s\}$ in a deterministic or random manner;
		\STATE Update $\by^k = (1+w_k)\tbx_{i}^{n_{i}^k} - w_k\tbx_{i}^{n_{i}^k-1}$
		\STATE Obtain $ \alpha_k $ via Algorithm \ref{alg:eststep}.
		\STATE Calculate $ \bz^k $ via \eqref{eq:updateBPG}.
		\IF{ \eqref{restartcri} holds }
		\STATE Set $\bx_{i}^{k+1} = \bz^k,\bx^{k+1}_j = \bx^k_j(j\neq i)$ and choose $w_{k+1}\in [0,\bar{w}]$.
		\ELSE
		\STATE Restart by setting $ X^{k+1} = X^k $ and $w_{k+1}=0$.
		\ENDIF
		\STATE  $ k = k+1 $.
		\ENDWHILE
	\end{algorithmic}
\end{algorithm}
The next lemma establishes the generalized descent property of the sequence $\{X^k\}$ generated by Algorithm~\ref{alg:BadaAPG}.
\begin{lemma}\label{lemma:de}
	Suppose Assumption \ref{assum1} and Assumption \ref{assum2} hold. Let $\{X^k\}$ be the sequence generated by Algorithm~\ref{alg:BadaAPG}. Then, we have $m_{k+1}\geq m_k$ and $\{E(X^{m_k})\}$ is non-increasing.
\end{lemma}
\begin{proof}
	By the definition of $m_k$, it is easy to know $m_{k+1}\geq m_k$.
	If the non-restart condition \eqref{restartcri} does not hold, we know $X^{k+1} = X^k$ which implies
	\begin{equation}\label{eq1}
	E(X^{m_k})-E(X^{k+1})\geq 0.
	\end{equation}
	If the non-restart condition \eqref{restartcri} holds, we have 
	\begin{equation}\label{eq2}
	E(X^{m_k})- E(X^{k+1})\geq\sigma\|X^k-X^{k+1}\|^2\geq0.
	\end{equation}
	Combing \eqref{eq1} with \eqref{eq2}, we obtain
	\begin{align*}
	E(X^{m_{k+1}})&=\max\{E(X^j)|[k+1-M]^*\leq j\leq k+1)\}\\
	&\leq  \max\{E(X^j)|[k-M]^*\leq j\leq k)\} = E(X^{m_k}).
	\end{align*}
\end{proof}
\begin{remark}
	Lemma \ref{lemma:de} implies that the line search approach ensures a general
	energy dissipation property associating with the constant $M\geq0$. For $M
	= 0$, the energy sequence $\{E(X^k)\}$ is monotone decreasing. For $M > 0$, 
	the Algorithm \ref{alg:BadaAPG} can find a controllable 
	subsequence with energy dissipation.
\end{remark}

\subsection{Convergence analysis}
In this subsection, we give a rigorous proof of energy and sequence convergence for Algorithm \ref{alg:BadaAPG}.
Since $ m_0 = 0 $, the sequence $ \{X^k\} $ generated by Algorithm \ref{alg:BadaAPG}
is contained in the sub-level set $ [E\leq E(X^0)] $. From the Assumption
\ref{assum1}, we know $ [E\leq E(X^0)] $ is compact. Together with Lemma
\ref{lem:welldefine}, we have $ \{X^k\} \subseteq [E\leq E(X^0)]\cap
\prod_{j=1}^s\calS_j \subseteq \calB(X^0)$ where $ \calB(X^0) $ is the closed
ball that contains $ [E\leq E(X^0)]\cap\prod_{j=1}^s\calS_j $. The next lemma
establishes that $E(X)$ is Lipschitz continuous on $\calB(X^0)$.

\begin{lemma}\label{lem:LipContinu}
	Suppose Assumption \ref{assum1} and Assumption \ref{assum2} hold. Then, there exists $ L_E >0$ such that $ E(X) $ is $ L_E $-Lipschitz continuous on $ \calB(X^0) $ for all $ X^0 $.
\end{lemma}
\begin{proof}
	Since $ \calB(X^0) $ is a compact subset of $ \bbC^N $, then $ \pi_j(\calB(X^0))\subseteq \bbC^{N_j} $ is closed. It's easy to know that $ \pi_j(\calB(X^0)) $ is bounded by the  fact that $ \|\bx_j\| \leq \|X\|  $. 
	Thus,  $ g_j $ is $ L_g^j $-Lipschitz continuous on $  \pi_j(\calB(X^0))
	\subseteq \ridom g_j$ (Corollary 8.41 in \cite{bauschke2011convex}). As a result, we have
	\begin{equation*}
	\begin{split}
	\left|\sum_{j= 1}^sg_j(\bx_j) -\sum_{j= 1}^sg_j(\by_j) \right| &\leq \sum_{j=1}^s |g_j(\bx_j) - g_j(\by_j)|	\leq \sum_{j=1}^s L_g^j\|\bx_j - \by_j\|\\
	&\leq s\left(\max_{j=1,2,\cdots s}L_g^j\right)\|X - Y\|,~\forall X,Y\in\calB(X^0).
	\end{split}	
	\end{equation*}
	Together with the fact that $ F\in C^1(\bbC^N) $, we conclude that there
	exists $ L_E >0$ such that $ E = F + \sum_{j=1}^sg_j $ is $ L_E $-Lipschitz continuous on the compact set $\calB(X^0)  $.
\end{proof}

\begin{lemma} \label{lem:lim_E}
	Suppose Assumption \ref{assum1} and Assumption \ref{assum2} hold. Let  $ \{X^k\}
	$ be the sequence generated by Algorithm~\ref{alg:BadaAPG}, there exists $ E^*>-\infty $ such that
	\begin{align}
	\lim\limits_{k\rightarrow\infty}\|X^{k+1} - X^k\| = 0,\quad \lim\limits_{k\rightarrow\infty} E(X^k) = E^*.
	\end{align}
\end{lemma}
\begin{proof} 
	We  show the proof similar to the framework in\,\cite{lu2013randomized}. 	
	Since $\{E(X^{m_k})\}$ is non-increasing from Lemma~\ref{lemma:de} and $E(X)$ is bounded below, 
	there exists $ E^* > -\infty $ such that $ 	\lim\limits_{k\rightarrow\infty}E(X^{m_k}) = E^* $.
	Define $ d^k =X^{k+1} - X^k $ and combine \eqref{eq1} with \eqref{eq2}, we have 
	\begin{align}\label{nonmonotonelinesearch}
	E(X^{m_k}) - E(X^{k+1}) \geq \sigma \|d^k\|.
	\end{align}
	Assume $k\geq M$, we prove that the following relations hold for any finite $ j\geq 1 $ by induction
	\begin{align}\label{eq:lim_d}
	\lim\limits_{k\rightarrow\infty}\|d^{m_k-j}\| =0,\quad \lim\limits_{k\rightarrow\infty} E(X^{m_k - j}) = E^*.
	\end{align}
	Substituting $ k $ by $ m_k-1 $ in \eqref{nonmonotonelinesearch}, we obtain 
	\begin{align*}
	\sigma\|d^{m_k-1}\|^2&\leq E(X^{m_{m_k-1}}) -E(X^{m_k})\leq E(X^{m_{k-M-1}}) -E(X^{m_k}),
	\end{align*}
	where the last inequality holds since $ m_k\geq k-M $. Let $k\to\infty$, we get
	\begin{align*}
	\lim\limits_{k\rightarrow\infty}\|d^{m_k-1}\| = 0.
	\end{align*}		
	Since $ E $ is $ L_E $-Lipschitz continuous on $\calB(X^0)  $ and  $ \{X^k\} \subseteq \calB(X^0)$, we have
	\begin{equation*}
	|E(X^{m_k-1}) - E(X^{m_k})|\leq L_E\|X^{m_k-1} - X^{m_k}\| = L_E\|d^{m_k-1}\|\to 0~(k\to\infty).
	\end{equation*}
	Thus, one has $ \lim\limits_{k\rightarrow\infty}E(X^{m_k-1}) =  \lim\limits_{k\rightarrow\infty}E(X^{m_k}) = E^*$, which implies that  \eqref{eq:lim_d} holds for $ j = 1 $. Suppose now that \eqref{eq:lim_d} holds for some $ j> 1 $, we show that it also holds for $ j+1 $. Substituting $ k $ by $ m_k-j-1 $ in \eqref{nonmonotonelinesearch}, it gives
	\begin{align*}
	\begin{split}
	\sigma\|d^{m_k-j-1}\|^2&\leq E(X^{m_{m_k-j-1}}) -E(X^{m_k-j})\leq E(X^{m_{k-M-j-1}}) -E(X^{m_k-j}).
	\end{split}
	\end{align*}
	Together with \eqref{eq:lim_d}, it means that 
	\begin{align*}
	\lim\limits_{k\rightarrow\infty} \|d^{m_k-j-1}\|=  0.
	\end{align*}
	Similarly, we have
	\begin{align*}
	\lim\limits_{k\rightarrow\infty}E(X^{m_k-j-1}) = \lim\limits_{k\rightarrow\infty}E(X^{m_k-j} - d^{m_k-j}) = \lim\limits_{k\rightarrow\infty}E(X^{m_k-j}) = E^*.
	\end{align*}
	Then \eqref{eq:lim_d} is also holds for $ j+1 $. By induction, we prove that relations \eqref{eq:lim_d} hold for any finite $ j\geq 1 $.
	
	From the definition of $ m_k $, we have the fact that
	\begin{align}\label{subcontain}
	m_{k+1} - m_k \leq M + 1,\quad \forall k\geq 0,
	\end{align}
	Thus, we obtain	$ \{X^k\}_{k=M+1}^\infty\subset \cup_{j=0}^{M+1}\{X^{m_k - j}\}_{k=M+1}^\infty $.
	Using \eqref{eq:lim_d}, it follows that $ \lim\limits_{k\rightarrow\infty}E(X^{m_k-j}) = E^*, \forall j\in[0,M]$. Then,  we have 
	\begin{align*}
	\lim\limits_{k\rightarrow\infty}E(X^{k}) = E^*.
	\end{align*}
	Moreover, we know
	\begin{align*}
	\lim\limits_{k\rightarrow\infty}\|d^k \|^2 \leq\dfrac{1}{\sigma} \lim\limits_{k\rightarrow\infty} \left(E(X^{m_k}) - E(X^k)\right)\rightarrow  0  \quad (k\to\infty).
	\end{align*}
	Then, we obtain $ \lim\limits_{k\rightarrow\infty}\|d^k \| = 0 $.
\end{proof}
\begin{remark}
	If $ M =0 $, Lemma \ref{lem:lim_E} can be easily obtained since $ m_k\equiv  k $ and we do
	not require to assume $ \pi_j(\calB(X^0))\subseteq \ridom g_i $ in Assumption \ref{assum1}.
\end{remark}

\begin{lemma}\label{lem:grad2}
	Suppose Assumption \ref{assum1}, Assumption \ref{assum2} and Assumption
	\ref{assum3} hold. Let $ \{X^k\} $
	be the sequence generated by Algorithm \ref{alg:BadaAPG}. Then, there exists a positive constant $ C$ such that
	\begin{align}\label{boundgrad}
	\dist(\vzero, \partial E(X^k))\leq  C\sum_{l=k-3T+1}^k\|X^{k} - X^{k-1}\|,\quad \forall k>3T,
	\end{align}
	where $ \dist(\vzero, \partial E(X^k)):=\inf \{\|y\|:y\in\partial E(X^k)\} $.
\end{lemma}
\begin{proof}
	If $X^k=X^{k-1}$, we only need to consider \eqref{boundgrad} holds at $
	X^{k-1} $ and it is easy to know $X^{k-2}\neq X^{k-1}$ from the monotonicity when $w_{k-1}=0$ as shown in Lemma~\ref{lem:Suffdescent}. Thus, we only consider the case $X^{k}\neq X^{k-1}$. 
	
	It is noted that $\partial E(X) = \nabla f(X) + U$ where
	$U=(\bu_1,\ldots,\bu_s)$, $ \bu_i\in\partial g_i(X)$, $1\leq i\leq s$.
	We first assume that the non-restart
	condition \eqref{restartcri} is satisfied at the $k$-th iteration.
	For each $ i\in\{1,2,\cdots, s\} $, we denote $ l_i^k$ as the last iteration at which the update of $ i $-th block is achieved within the first $ k $-th iteration, i.e., $l_i^k = \argmax\{\ell|b_{\ell} = i,\ell\leq k\}$. 	
	Note that $ l_{b_k}^k = k $ and $ \bx_i^{l_i^k} = \tbx_i^{n_i^k} $. 
	By the optimal condition of the proximal subproblem  \eqref{Bprox}, we have
	\begin{align}\label{eq:suboptimality}
	\vzero \in \partial g_i(\bx^{l_i^k}) + \nabla_i f(\by^{l_i^k-1},\bx_{\neq i}^{l_i^k-1}) + \dfrac{1}{\alpha_{l_i^k-1}}\left(\nabla h_i(\bx^{l_i^k}_i) - \nabla h_i(\by^{l_i^k-1})\right).
	\end{align}
	Since that $ \calB(X^0) $ is compact, we let  $ \rho_h := \max\limits_{j=1,2,\cdots
		s}(\max_{X\in\calB(X^0)} \|\nabla^2 h_j(X)\| )$ and  $ \rho_f $ be the local
	Lipschitz constant of $ \nabla f $ on  $ \calB(X^0) $. Due to $ \{X^k\}\subseteq \calB(X^0) $, we know
	\begin{align}\label{eq:ExtraStep}
	\by^{l_i^k-1}= \tbx_i^{n_i^k-1} + w_{l_i^k-1}(\tbx_i^{n_i^k-1} -\tbx_i^{n_i^k-2} )\in\calB_0(X^0).
	\end{align}
	Together with \eqref{eq:suboptimality} and \eqref{eq:ExtraStep}, we get
	\begin{equation}\label{bound3}
	\begin{aligned}
	&~\inf_{\bu_i\in  \partial g_i(\bx_i^k)}\|\nabla_i f(X^k) +\bu_i\|\\  
	\leq&~ \|\nabla_i f(X^k) - \nabla_i f(\by^{l_i^k-1}, \bx_{\neq i}^{l_i^k-1}) - \dfrac{1}{\alpha_{l_i^k-1}}\left(\nabla h_i(\bx_i^{l_i^k}) - \nabla h_i(\by^{l_i^k-1})\right)\|\\
	\leq&~ \|\nabla_i f(X^k) - \nabla_i f(\by^{l_i^k-1}, \bx_{\neq i}^{l_i^k-1})\| + \dfrac{1}{\alpha_{l_i^k-1}}\|\nabla h_i(\bx_i^{l_i^k}) - \nabla h_i(\by^{l_i^k-1})\|  \\
	\leq&~ \rho_f \|X^k - (\by^{l_i^k - 1}, \bx_{\neq i}^{l_i^k - 1})\| + \dfrac{\rho_h}{\alpha_{\min}}\|\bx_i^{l_i^k} - \by_i^{l_i^k - 1}\|.
	\end{aligned}
	\end{equation}
	If $ i = b_k $, it follows that $ l_i^k = k  $. Then, we have 
	\begin{equation}\label{bound1}
	\|X^k - (\by^{l_i^k - 1}, \bx_{\neq i}^{l_i^k - 1})\| = \|X^{l_i^k} - (\by^{l_i^k - 1}, \bx_{\neq i}^{l_i^k - 1})\| = \|\bx_i^{l_i^k} - \by_{i}^{l_i^k - 1}\|,\quad i = b_k.
	\end{equation}
	If $ i \neq b_k $, we know
	\begin{equation}\label{bound2}
	\begin{aligned}
	\|X^k - (\by^{l_i^k - 1}, \bx_{\neq i}^{l_i^k - 1})\|
	\leq&~   \sum_{l = l_i^k+1 }^k\|X^l - X^{l- 1}\|  + \|X^{l_i^k} - (\by^{l_i^k-1}, \bx_{\neq i}^{l_i^k-1})\|\\
	=&~   \sum_{l = l_i^k+1 }^k\|X^l - X^{l- 1}\|  +\|\bx_i^{l_i^k} - \by_i^{l_i^k-1}\|.
	\end{aligned}
	\end{equation}
	Combining \eqref{bound3}, \eqref{bound1} and \eqref{bound2}, we have
	\begin{align*}
	\inf_{\bu_i\in  \partial g_i(\bx_i^k)}\|\nabla_i f(X^k) +\bu_i\|\leq
	\begin{cases}
	\rho_1\|\bx_i^{l^k_i} - \by_i^{l_i^k-1}\|,& i = b_k,\\
	\rho_f\sum_{l = l_i^k+1 }^k\|X^l - X^{l- 1}\| + \rho_1\|\bx_i^{l^k_i} - \by_i^{l_i^k-1}\|,& i \neq  b_k,
	\end{cases}
	\end{align*}
	where $ \rho_1 = \rho_f + \rho_h/\alpha_{\min} $. Moreover, for each $ i $, we
	know that $ k-l_i^k\leq T $ by the Assumption \ref{assum3} and there exist $
	k_i^1,k_i^2\in[k-3T, k] $ such that $ 	\bx_i^{k_i^1} = \tbx_i^{n_i^k-1},
	\bx_i^{k_i^2} = \tbx_i^{n_i^k-2} $. Thus,  we have $\|\bx_i^{l_i^k} -
	\by_i^{l_i^k-1}\|\leq\|\tbx_i^{n^k_i} - \tbx_i^{n^k_i-1}\| +
	\bar{w}\|\tbx_i^{n^k_i-1} - \tbx_i^{n_i^k-2}\|$ and
	\begin{align*}
	&~ \dist(\vzero,\partial E(X^k)) 
	= \inf_{U\in\partial g_i}\left(\sum_{i=1}^s \|\nabla_i F(X^k) +\bu_i\|^2\right)^{1/2}
	\leq  \sum_{i=1}^s \inf_{\bu_i\in\partial g_i} \|\nabla_i F(X^k) +\bu_i\|\\
	\leq&~ \rho_f \sum_{i\neq b_k} \sum_{l = l_i^k+1 }^k\|X^l - X^{l- 1}\|+\rho_1\sum_{i=1}^s\|\bx_i^{l^k_i} - \by_i^{l_i^k-1}\|\\
	\leq&~ \rho_f (s-1)\sum_{l = k-3T+1 }^k\|X^l - X^{l- 1}\| +\rho_1\sum_{i=1}^s\left(\|\tbx_i^{n^k_i} - \tbx_i^{n^k_i-1}\| + \bar{w}\|\tbx_i^{n^k_i-1} - \tbx_i^{n_i^k-2}\|\right)\\
	\leq&~  \rho_f (s-1)\sum_{l = k-3T+1 }^k\|X^l - X^{l- 1}\|+(1+\bar{w})\rho_1\sum_{l = k-3T+1 }^k\|X^l - X^{l- 1}\|.
	\end{align*}
	Define $ C :=  (s-1)\rho_f + (1+\bar{w})\rho_1 $, we obtain \eqref{boundgrad}.
\end{proof}

Together with Lemma \ref{lem:lim_E} and Lemma \ref{lem:grad2}, we immediately have the following sub-sequence convergent property.
\begin{thm}\label{thm:subseqConvergence}
	Suppose Assumption \ref{assum1}, Assumption \ref{assum2} and Assumption
	\ref{assum3} hold. Let $ \{X^k\}  $ be the sequence generated by Algorithm \ref{alg:BadaAPG}, then any limit point $ X^* $ of $ \{X^k\} $ is a critical point of $ E $, i.e., $ \vzero \in\partial E(X^*) $.
\end{thm}
\begin{proof}
	From Remark \ref{rmk:lem2}, we know $ \{X^k\}\subset [E\leq
	E(X^0)] $ and thus is bounded. Then, the set of limit points of $ \{X^k\} $ is
	nonempty. For any limit point $ X^* = (\bx_1^*,\bx_2^*,\cdots,\bx_s^*) $, there
	exists a subsequence $ \{X^{k_j}\} $ such that $ \lim\limits_{j\to\infty} X^{k_j}
	= X^* $. By Lemma \ref{lem:lim_E} and Lemma \ref{lem:grad2}, we immediately obtain
	\begin{align}\label{subgralim}
	\lim\limits_{j\to\infty}\dist(\vzero, \partial E(X^{k_j})) = 0.
	\end{align}
	Moreover, from Lemma \ref{lem:LipContinu}, we know $ E(X) $ is $ L_E $-Lipschitz smooth on $ \calB(X^0) $. Using  the fact that $ \{X^k\}\subseteq \calB(X^0) $ and $ \calB(X^0) $ is compact, we know $ X^*\in \calB(X^0) $. Thus, we get $ E(X^{k_j})\to E(X^*) $ as $ j\to\infty $.	
	For any $ u_j\in\partial E(X^{k_j}) $, we know
	\begin{equation*}
	E(X) \geq E(X^{k_j}) + \langle u_j, X - X^{k_j}\rangle,\quad \forall X\in\dom E,
	\end{equation*}
	which implies $ \{u: u = \lim\limits_{j\to\infty} u_j, u_j\in\partial E(X^{k_j}) \}\subset \partial E(X^*) $ by setting $ j\to\infty $. Together with \eqref{subgralim}, we conclude that  $ \vzero\in \partial E(X^*) $.
\end{proof}

When $ M = 0 $ and $E$ is a KL function~\cite{bolte2014proximal}, the sub-sequence convergence can be strengthen to the whole sequence convergence.
\begin{definition}[KL function\,\cite{bolte2014proximal}]
	\label{assum4}
	$ E(x) $ is the $ KL $ function if for all $ \bar{x}\in \dom \partial E  := \{x :
	\partial E(x)\neq \emptyset\}$, there exists $\eta>0$, a neighborhood $ U $ of $
	\bar x $ and $\psi\in\Psi_\eta: =\{\psi\in C[0,\eta)\cap C^1(0,\eta)$, where
	$\psi$ is concave, $\psi(0)=0$, $\psi^{'}>0$  on $(0,\eta)\}$ such that for
	all $x\in U\cap \{x:E(\bar x)<E(x)<E(\bar x)+\eta\}$, the following inequality holds,
	\begin{equation}
	\psi^{'}(E(x)-E(\bar x))\,\dist(\vzero,\partial E(x))\geq 1.
	\end{equation}
\end{definition}

\begin{thm}[Sequence convergence]\label{thm:SeqConvergence}
	Suppose Assumption \ref{assum1}, Assumption \ref{assum2}, Assumption \ref{assum3} hold. Let  $ \{X^k\}  $ be the sequence generated by Algorithm \ref{alg:BadaAPG} with $ M = 0 $. If $E$ is a KL function, then there exists some $X^*$ such that
	\begin{align}
	\lim\limits_{k\rightarrow\infty} X^k = X^*, \quad \vzero\in\partial E(X^*).
	\end{align}
\end{thm}
\begin{proof}
	The proof is in the Appendix B.
\end{proof}

In the following context, we introduce two kinds of $h$ and present corresponding
numerical algorithms for solving \eqref{finitepro}.
	
	\section{Application to the CMSH model}\label{sec:appCMSH}
As discussed above, let $ f = F $, $ g_i = G_i  $ and $ \calS_i =
\{\bhphi:e_1^\top\bhphi = 0\} $, then the problem \eqref{finitepro} reduces to
\eqref{GeneralFormulation}. In this section, we apply the Algorithm \ref{alg:BadaAPG} to solve the
finite dimensional CMSH model \eqref{finitepro}. Let $ i = b_k $, a key component of 
efficiently implementing AB-BPG method is fast solving the following constrained subproblem
\begin{align}
\label{BProxsubprob}
\min_{\bhphi\in\calS_i}~ G_i(\bhphi) + \langle \nabla_i F(\bhpsi^k, \bhphi_{\neq i}^k), \bhphi - \bhphi_i^k\rangle + \dfrac{1}{\alpha_k}D_{h_i}(\bhphi, \bhphi_i^k)
\end{align}
where  $ \bhpsi^k = (1+w_k)\bhphi_i^k - w_k\bhphi_i^{\text{prev}} $ is the extrapolation and
$ \bhphi_i^{\text{prev}} $ is the value of $ \bhphi_i $ before it is updated to $
\bhphi_i^k $. 
Solving \eqref{BProxsubprob} depends on the form of $h_i$.
In the following context, we let
\begin{equation}\label{kernel}
h_j(\bx) := h(\bx) = \dfrac{a}{4}\|\bx\|^4 + \dfrac{1}{2}\|\bx\|^2,  \quad a\geq 0
\end{equation}
for all $ j $ and propose two classes of numerical algorithms for solving
\eqref{infpro} based on different choices of $a$. For each algorithm, we will show that the subproblem \eqref{BProxsubprob} is well defined and can be solved efficiently.

{\textbf{Case I: $a = 0$.}}
In this case, $h_j(\bx)=\|\bx\|^2/2$ and the Bregman divergence of $D_h$ becomes the Euclidean distance, i.e.,\
\begin{equation}
D_h(\bx,\by) = \dfrac{1}{2}\|\bx-\by\|^2.
\end{equation}
Thus, the iteration scheme \eqref{BProxsubprob} is reduced to the accelerated block proximal gradient method\,\cite{bao2015dictionary,xu2017globally}. We can find the closed-form solution of the constrained minimization problem \eqref{BProxsubprob} by applying Lemma 4.1 in \cite{jiang2020efficient}.
\begin{lemma}\label{lem:K=2}
	Given $ \alpha_k>0 $, $ \Phi^k $ and $ \bhpsi^k \in\calS_i$, the subproblem \eqref{BProxsubprob} with $ h(\bx) = \|\bx\|^2/2 $ is well-defined and has analytical solution
	\begin{align}\label{BProx_Norm2}
	\begin{split}
	\bhphi_i^{k+1}
	&=\left(\alpha_k \calD_i+ I\right)^{-1}\left(\bhpsi^k  -\alpha_k\calP_1\nabla_i F(\bhpsi^k, \bhphi_{\neq i}^k)\right) 
	\end{split}
	\end{align}
	where $ \calD_i $ is a $N_i$-order diagonal matrix as defined in \eqref{def:Matrix_D} and $ \calP_1 = I-e_1e_1^\top $ is
	the projection onto $ \calS_i $.
\end{lemma}
From the feasibility assumption, it is noted that $ \bhpsi^k\in \calS_i $ holds as
long as the initial point $ \Phi^0\in\prod_{j=1}^s\calS_j $. The concrete algorithm
is given in Algorithm \ref{alg:ABBAPG-K-PFC} with $ K = 2 $.

{\textbf{Case II: $a>0$.}}
Since $ F(\hPhi) $ can be represented as a
$ 4^{th} $-degree polynomial function, it is known that $F$ is not block wise relatively smooth with respect to $ h(\bx) = \|\bx\|^2/2 $. In this case, we choose $ a, b>0 $.
The next lemma shows the optimal condition of minimizing \eqref{BProxsubprob}, which can be obtained from Lemma 4.2 in \cite{jiang2020efficient}.
\begin{lemma}
	Given $ \alpha_k>0 $, $ \Phi^k $ and $ \bhpsi^k \in\calS_i$,  the subproblem
	\eqref{BProxsubprob} with $
	h(\bx)=\dfrac{a}{4}\|\bx\|^4+\dfrac{b}{2}\|\bx\|^2+1~
	(a,b >0) $ is well-defined and has the closed form as follows
	\begin{align}\label{BProx_Norm4}
	\bhpsi_i = [\alpha_k \calD_i + (ap^*+b)I]^{-1}\left(\nabla h(\bhpsi^k) - \alpha_k \calP_1\nabla F(\bhpsi^k, \bhphi_{\neq i}^k)\right),	
	\end{align}
	where $ \calD_i $ is given in \eqref{def:Matrix_D} and $ p^* $ is the fixed point of $  p =\|\bhphi_i^{k+1}\|^2 :=  r(p) $.
\end{lemma}
It is noted that the iterate \eqref{BProx_Norm4} requires solving a nonlinear scalar equation, which can efficiently be solved by many existing solvers. 
In our implementation, the Newton method is used.
The concrete algorithm is given in Algorithm \ref{alg:ABBAPG-K-PFC} with $ K = 4 $.
\begin{algorithm}[!pbht]
	\caption{AB-BPG-K method for PFC model}
	\label{alg:ABBAPG-K-PFC}
	\begin{algorithmic}[1]
		\REQUIRE $\hPhi^0=\hPhi^{-1}\in\prod_{j=1}^s \calS_j$, $\rho\in (0,1)$, $ \sigma\geq\eta  >0$ and $ w_0 = 0 $, $\bar w,\alpha_0, M>0$, $ k = 0 $.
		\WHILE {stopping criterion is not satisfied}
		\STATE Pick $ i = b_k \in\{1,2,\cdots, s\}$ in a deterministic or random manner
		\STATE Update $\bhpsi^k = (1+w_k)\bhphi_{i}^{k} - w_k\bhphi_{i}^{pre}$
		\STATE Estimate $ \alpha_k $ by Algorithm \ref{alg:eststep}
		\IF {$K=2$}
		\STATE Calculate $\bz^k=\left(\alpha_k \calD_i + I\right)^{-1}\left(\bhpsi^k  -\alpha_k\calP_1\nabla_i F(\bhpsi^k, \bhphi_{\neq i}^k)\right)$
		\ELSIF {$K=4$}
		\STATE Calculate the fixed point of  \eqref{BProx_Norm4}.
		\STATE Calculate  $\bz^k=	[\alpha_k \calD_i + (ap^*+b)I]^{-1}\left(\nabla h(\bhpsi^k) - \alpha_k \calP_1\nabla_i F(\bhpsi^k, \bhphi_{\neq i}^k)\right)$
		\ENDIF
		\IF {$E(\hPhi^{m_k})-E(\bz^{k}, \bhphi_{\neq i}^k)\geq \sigma \|\bhphi_i^k-\bz^{k}\|^2$ }
		\STATE $ \bhphi_i^{k+1} = \bz^k $, $ \bhphi_j^{k+1} = \bhphi_j^k~ (j\neq i)$ and choose $w_{k+1}\in [0,\bar{w}]$.
		\ELSE
		\STATE Restart by setting $ \hPhi^{k+1} = \hPhi^k $ and  $w_{k+1}=0$.
		\ENDIF	
		\STATE $ k =k+1 $.
		\ENDWHILE	
	\end{algorithmic}
\end{algorithm}

\subsection{Convergence analysis of Algorithm \ref{alg:ABBAPG-K-PFC}}
The convergence analysis of Algorithm \ref{alg:BadaAPG} can be directly applied for Algorithm \ref{alg:ABBAPG-K-PFC} if
all required assumptions in Theorem \ref{thm:subseqConvergence} are satisfied.
We first show that the energy function $E$ of CMSH model satisfies Assumption~\ref{assum1}, Assumption~\ref{assum3}. Then,  Assumption \ref{assum2} is analyzed for Case (P2) and Case (P4) independently. 
\begin{lemma}\label{lem:PFCproperty1}
	Let $ E(\hPhi) = F(\hPhi) + \sum_{j=1}^sG_j(\bhphi_j) $ be the energy
	function of \eqref{Discret}.  Then, it satisfies
	\begin{enumerate}
		\item $E$ is bounded below and level bounded,
		\item  $ \ridom G_i = \bbC^{N_j} $, thus $ \pi_j(\calB(\hPhi^0)) \subseteq \ridom G_i $ for all $ j $.			
	\end{enumerate}
\end{lemma}
\begin{proof}
	From the continuity and the coercive property of $ F $, i.e., $ F(\hPhi)\to +\infty $ as $ \hPhi\to \infty $, the sub-level set $ [E \leq \alpha] $ is compact for any $ \alpha\in\bbR $. 
	Since $ \dom G_i = \bbC^{N_j} $, we directly get that $ \ridom G_j =  \bbC^{N_j} $.
\end{proof}

\begin{lemma}\label{lem:PFCproperty2}
	Let $F(\hPhi)$ be defined in \eqref{defined_GandF}. Then, we have 
	\begin{enumerate}
		\item If $h$ is chosen as case (P2), then $F$ is block-wise relative smooth with respect to $h_i\equiv h$ in any compact set $[E\leq E(X^0)]$.
		\item If $h$ is chosen as the case (P4), then $F$ is block-wise relatively smooth to $h_i\equiv h$.
	\end{enumerate}
\end{lemma}
\begin{proof}
	Denote $\hPhi^{\otimes k}:= \hPhi\otimes\hPhi\otimes\cdots\otimes\hPhi$
	where $\otimes$ is the tensor product. Then, $F(\hPhi)$ is the $4^{th}$-degree polynomial,
	i.e., $F(\hPhi) = \sum_{k=2}^4\langle\mathcal{A}_k,\hPhi^{\otimes k}\rangle$
	where the $k^{th}$-degree monomials are arranged as a $k$-order tensor $\mathcal{A}_k$.
	For any compact set $[E\leq E(X^0)]$, $\nabla F$ is bounded and thus $F$ is
	block-wise relative smooth with respect to any polynomial function in $[E\leq
	E(X^0)]$ which includes case (P2). Moreover, for any fixed $ \bhphi_{\neq j}
	$, $ F_j(\bhpsi)  =F(\bhpsi, \bhphi_{\neq j}) $ is still a $ 4^{th} $-degree
	polynomial. When $h$ is chosen as (P4), according to Lemma 2.1 in\,\cite{li2019provable}, there exists $R_F^j>0$ such that $F_j(\bhpsi)$ is $R_F^j$-smooth relative to $ h $. 
\end{proof}
Combining Lemma \ref{lem:PFCproperty1}, Lemma \ref{lem:PFCproperty2} with Theorem
\ref{thm:subseqConvergence}, we can directly give the convergence analysis of
Algorithm \ref{alg:ABBAPG-K-PFC}.
\begin{thm}\label{thm:SeqConvergence-PFC}
	Let $ E(\hPhi)  = F(\hPhi) + \sum_{j=1}^sG_j(\bhphi_j)$ be the energy
	function which is defined in \eqref{Discret}. Then any limit point $ \hPhi^* $ of
	$ \{\hPhi^k\} $ is a critical point of E, i.e.,  $ \nabla E(\hPhi^*) = 0 $, if
	the sequence $\{\hPhi^k\}$ generated by Algorithm \ref{alg:ABBAPG-K-PFC}
	satisfies one of the following conditions:
	\begin{enumerate}
		\item  $K=2$ in Algorithm \ref{alg:ABBAPG-K-PFC} and  $\{\hPhi^k\}$ is bounded.
		\item  $K=4$ in Algorithm \ref{alg:ABBAPG-K-PFC}. 
	\end{enumerate}		
\end{thm}
It is noted that when $h$ is chosen as (P2), we cannot bounded the growth of $F$ as $F$ is a fourth order polynomial. Thus, the boundedness assumption of $\{\hPhi^k\}$ is imposed which is similar to the requirement as the semi-implicit scheme\,\cite{shen2010numerical}.

If $ M = 0 $, Theorem~\ref{thm:SeqConvergence} implies that the sub-sequence convergence can be strengthened by the requirement of the KL property of $ E $. According to Example 2 in\,\cite{bolte2014proximal}, it is easy to know that $ E(\hPhi) $ in our model is a semi-algebraic function, then it is a KL function by Theorem 2 in\,\cite{bolte2014proximal}. 
Thus, the sequence convergence of Algorithm~\ref{alg:ABBAPG-K-PFC} with $M=0$ is obtained as follows.
\begin{thm}
	Let $ E(\hPhi)  = F(\hPhi) + \sum_{j=1}^sG_j(\bhphi_j)$ be the energy
	function  defined in \eqref{Discret} and $ M = 0 $ in Algorithm
	\ref{alg:ABBAPG-K-PFC}. If  the sequence $\{\hPhi^k\}$ generated by Algorithm\ref{alg:ABBAPG-K-PFC} satisfies one of the conditions in Theorem \ref{thm:SeqConvergence-PFC}, then there exists some $\hPhi^*$ such that $\lim\limits_{k\to\infty}\hPhi^k=\hPhi^*$ and $\nabla E(\hPhi^*) = 0$.
\end{thm}
From the above analysis, we know that the proposed AB-BPG method has the proven convergence to some stationary point compared to the current gradient flow based methods. Moreover, it has shown that the generated sequence has the (generalized) energy dissipation and mass conservation properties.

\section{Numerical results}\label{sec:results}
In this section, we apply the AB-BPG-K approaches (Algorithm~\ref{alg:ABBAPG-K-PFC}) to binary, ternary, quinary 
component systems based on the CMSH model. 
The efficiency and accuracy of our 
methods are demonstrated through comparing with existing methods, including the
first-order semi-implicit scheme (SIS), the second-order Adam-Bashforth with Lagrange
extrapolation approach (BDF2)\,\cite{jiang2020stability}, the scalar auxiliary variable (SAV)
method\,\cite{shen2018scalar}, the stabilized scalar auxiliary variable (S-SAV). Note that these
employed methods all guarantee the equality constraint, i.e., mass conservation. 

The step sizes $ \alpha_k $ in the AB-BPG approaches are adaptively obtained via the linear search technique. To be fair, adaptive time stepping are applied to gradient flow methods. The SIS and BDF2 scheme use the adaptive time stepping\,\cite{zhang2013adaptive}
	    \begin{equation}\label{eq:adpstep1}
	        \alpha_k = \max\left\{ \alpha_{\min}, \frac{\alpha_{\max}}{\sqrt{1 + \rho |E'(\Phi^k)|^2}}\right\}
	    \end{equation}
where $E$ is the energy functional defined in the model. The constant $\alpha_{\min},\alpha_{\max},$ and $\rho$ are set to $\alpha_{\min} = 0.001, \alpha_{\max} = 0.1$ and $\rho =50$ as\,\cite{zhang2013adaptive} suggested. The adaptive SAV scheme is implemented as Algorithm 1 in\,\cite{shen2019new}, where a semi-implicit first-order SAV scheme and a semi-implicit second-order SAV scheme based on Crank-Nicolson  are computed at one step iteration. The adaptive time stepping is in the form of 
\begin{equation}\label{eq:adpstep2}
    \alpha_{k+1} = \max\left\{\alpha_{\min},\min\left\{\rho\left(\frac{\mathrm{tol}}{e_{k+1}}\right)^{1/2}\alpha_k,\alpha_{\max}\right\}\right\}
\end{equation}
where $\rho$ is a default safety coefficient, $\mathrm{tol}$ is a reference tolerance, $e_{k+1}$ is the relative error between a first-order SAV scheme and a second-order SAV scheme at each time level. $e_{n+1}$ is the relative error between first-order SAV scheme and second-order scheme at each time level. The original parameters ($\rho = 0.9$, $\mathrm{tol} = 10^{-3}$, $\alpha_{\min} = 10^{-5}$, $\alpha_{\max} = 10^{-2}$) taken in work\,\cite{shen2019new} are inefficient to 
compute stationary states. In our implementation, we carefully choose appropriate parameters 
case by case to ensure better numerical performance.
Compared to the adaptive SAV scheme, the S-SAV scheme adds a first-order stabilization term $S_1(\Phi^{k+1}-\Phi^{k})$ and a second-order stabilization term $S_2(\Phi^{k+1}-2\Phi^{k} + \Phi^{k-1})$\,\cite{shen2010numerical} to the SAV-SI scheme and the SAV-CN scheme, respectively. In our computation,                 we let $S_1 = S_2 = 10$  and  use the adaptive times stepping \eqref{eq:adpstep2} with $\rho = 0.9$, $\mathrm{tol} = 10^{-3}$, $\alpha_{\min} = 10^{-5}$, $\alpha_{\max} = 1$. 

Our methods, the adaptive SIS and the adaptive BDF2 use a Gauss-Seidel manner to update order parameters in a fixed cyclic order. While the adaptive SAV and adaptive S-SAV method keep the same as Algorithm 1 in\,\cite{shen2019new}, which uses a Jacobian manner to update all order parameters simultaneously in one step iteration. 
In our implementation, all approaches are stopped when $\|\nabla E(\hPhi)\|_\infty < 10^{-7}$  or the energy difference between two iterations is less than $10^{-14}$.
All experiments were performed on a workstation with a 3.20 GHz CPU (i7-8700, 12 processors). All codes were written by MATLAB without parallel implementation.

\subsection{ Binary component systems}
We first choose $ s = 2 $  in \eqref{eq:energy} and take the two-dimensional decagonal quasicrystal as an example to examine our approaches' performance. The decagonal quasicrystal can be embedded into a four-dimensional periodic structure. Therefore, we carry out the projection method in four-dimensional space. The $4$-order invertible matrix $\bB$ associated with the  four-dimensional periodic structure is chosen as $\bI_4$.  
The corresponding computational domain in physical space is $[0,2\pi)^4$. 
The projection matrix $\calP$ in \eqref{eq:pm} of the decagonal quasicrystals is
\begin{equation}
\mathcal{P} =\left(
\begin{array}{cccc}
1 & \cos(\pi/5) & \cos(2\pi/5) & \cos(3\pi/5) \\
0 & \sin(\pi/5) & \sin(2\pi/5) & \cos(3\pi/5) 
\end{array}
\right).
\label{eqn:DQC:projMatrix}
\end{equation} 
We use $38^4$ plane wave functions to discretize the binary CMSH energy functional. The
parameters in \eqref{eq:energy} are given in Table \ref{tab:DQCparas}. The initial
configuration of order parameters is chosen as references \cite{jiang2016stability,jiang2020efficient} suggest. 
The stationary quasicrystals, including physical space morphology and Fourier
spectra, are given in Figure \ref{fig:DQCPhase}.

\begin{table}[!htbp]
	\centering
	\caption{The non-zero model parameters used in computing binary
		decagonal quasicrystal.}
	\begin{tabular}{|c|}
		\hline
		\begin{tabular}[c]{@{}c@{}}  $ c = 20 $, $ q_1 =1 $, $ q_2 =2\cos(\pi/5) $, $ \tau_{0,2} = \tau_{2,0} = -0.1 $, $ \tau_{0,3} = \tau_{3,0} = -0.3 $, \\ $ \tau_{1,2} =  \tau_{2,1} = -2.2 $,$ \tau_{0,4} = \tau_{4,0} =\tau_{1,1} =  \tau_{2,2} =   \tau_{1,3} =  \tau_{3,1} = 1 $.
		\end{tabular} \\ \hline
	\end{tabular}
	\label{tab:DQCparas}
\end{table}

\begin{figure}[!htbp]
	\centering
	\subfigure[]{\includegraphics[scale=0.2]{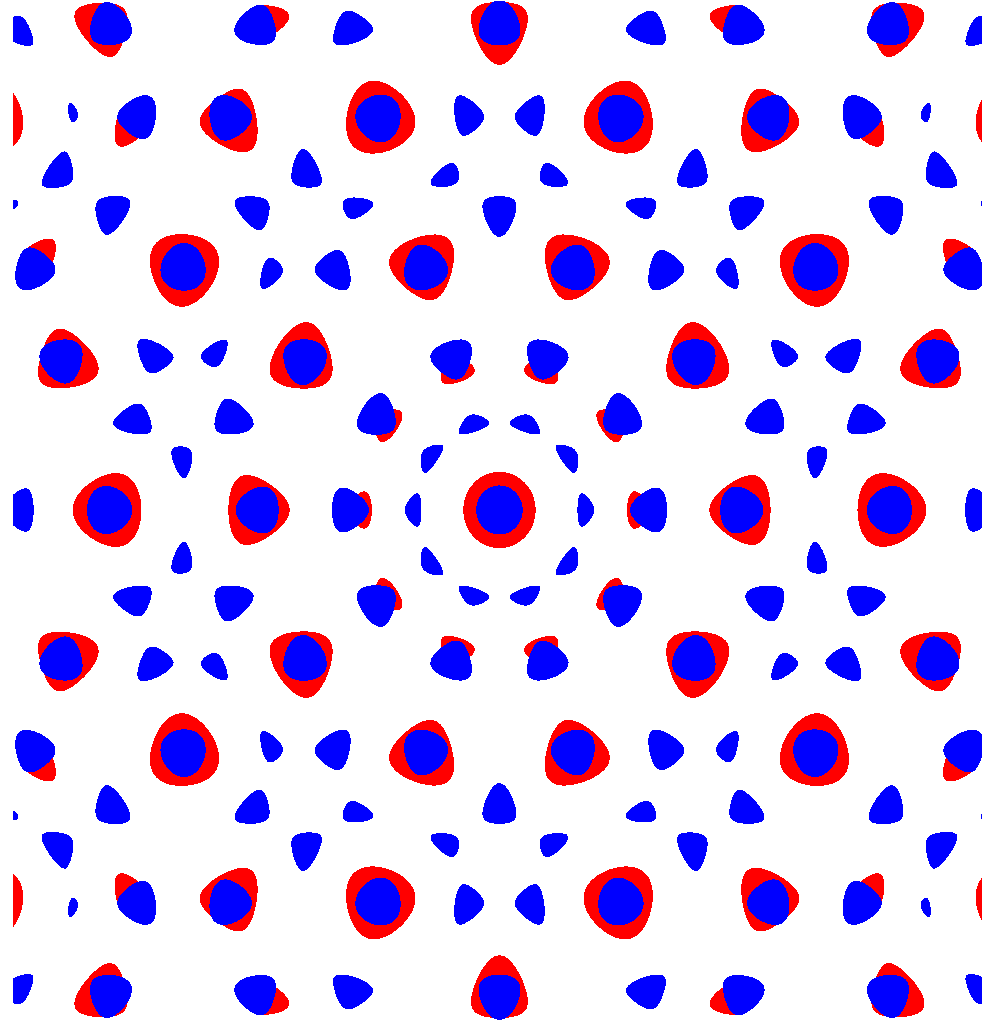}}
	\subfigure[]{\includegraphics[scale=0.25]{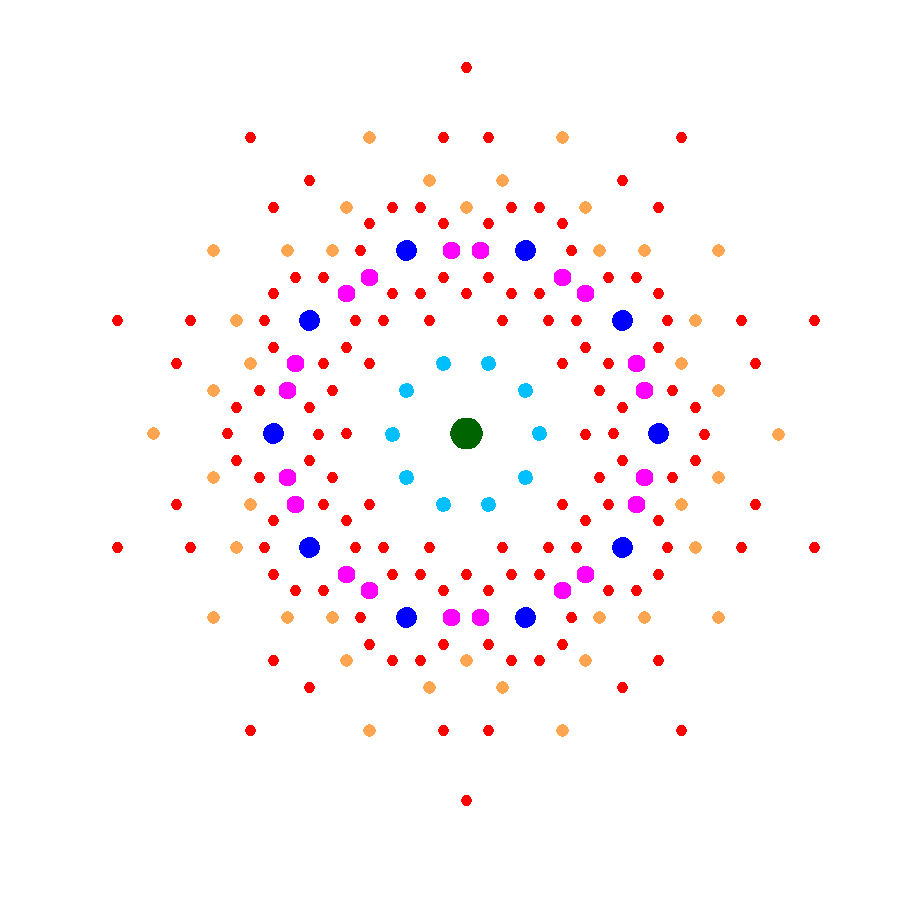}}
	\subfigure[]{\includegraphics[scale=0.25]{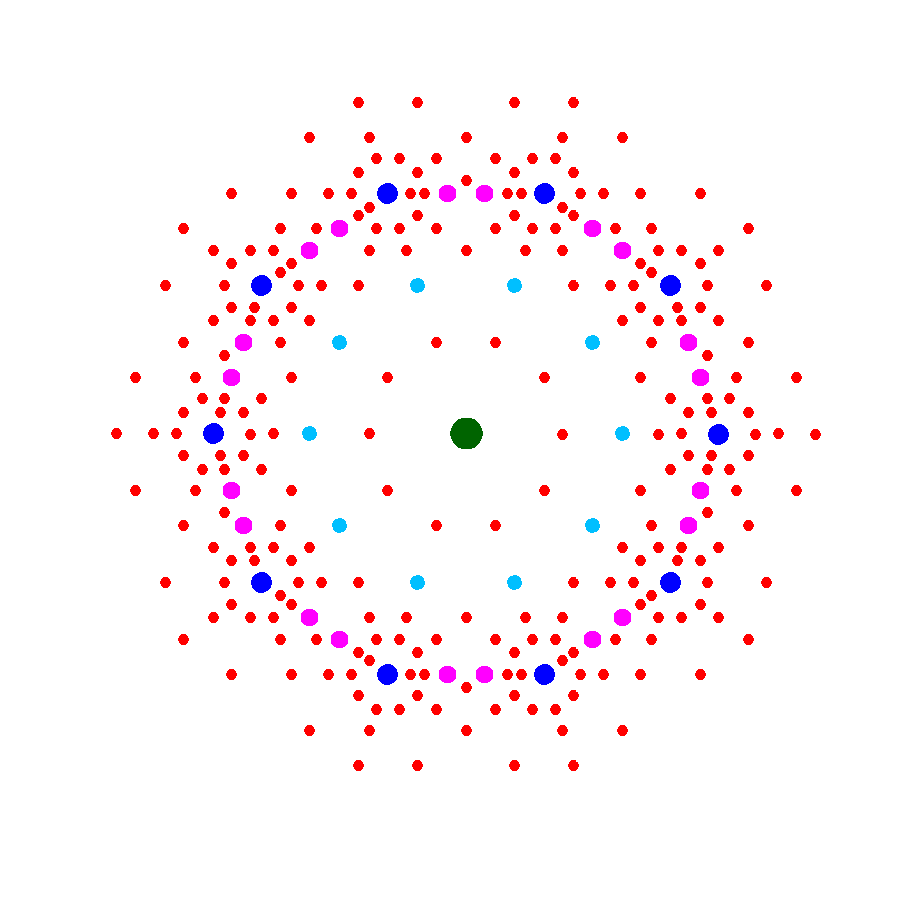}}
	\caption{The stationary decagonal quasicrystal in the binary CMSH model. \textbf{(a):} The
		spatial distribution of density, where the red and blue colors correspond to the
		rich concentration of $ \phi_1 $ and $ \phi_2 $, respectively;
		\textbf{(b):} The Fourier spectra of $ \phi_1 $; \textbf{(c):} The
		Fourier spectra of $ \phi_2 $. Only Fourier spectral points whose
		coefficient intensities are larger than $10^{-4}$ are presented.
	}
	\label{fig:DQCPhase}
\end{figure}

\subsubsection{Algorithm study}
In this subsection, we take the binary CMSH model as an example to show our proposed AB-BPG method's performance by choosing different hyperparameters, including the choice of the step sizes, the descent subsequence, and the block update manner. Similar results exist in the following ternary and quinary cases. For simplicity, we only present the results in the computing binary CMSH model.

The step size of the AB-BPG methods is adaptively obtained by Algorithm~\ref{alg:eststep}. In our implementation, we set $\alpha_0 = 0.1$, $\varsigma = (\sqrt{5} - 1)/2$, $\eta = 10^{-12}$, $\alpha_{\min} = 10^{-6}$ and $\alpha_{\max} = 10$ in Algorithm~\ref{alg:eststep}. Figure~\ref{fig:AlgorithmStudy} (a) illustrates the adaptive step sizes versus the iteration when $M=0$ and $a=1$. 

In Figure~\ref{fig:AlgorithmStudy} (b), it shows the total energy of the sequence $\{\hPhi^k\}$ and the subsequence $\{\hPhi^{m_k}\}$ defined in \eqref{def:m_k} versus the iterations with $M=5$ and $a = 0$. It is observed that the subsequence $\{E(\hPhi^{m_k}\})$ is monotone decreasing which validates the generalized energy dissipation property proved in Lemma \ref{lemma:de}.

In the next, we consider the different choices of the update blocks. One is the cyclic rule that updates $\hPhi_1$ and $\hPhi_2$ alternatively in the Gauss-Seidel fashion. The other is to randomly choose the update block, and $\hPhi_1$ and $\hPhi_2$ should be updated at least once in $ 10 $ consecutive iterations. In this case, we set $a=M=0$, and we independently run the random update rule 50 times and report its terminated iteration numbers for each trial. The result is shown in Figure~\ref{fig:AlgorithmStudy} (c). Compared to the random update manner, it is observed that the cyclic rule needs fewer iterations for the desired accuracy. Therefore, in the following simulations, we only consider the cyclic update order when carrying out the AB-BPG algorithms.
\begin{figure}
	\centering
	\subfigure[]{\includegraphics[scale=0.25]{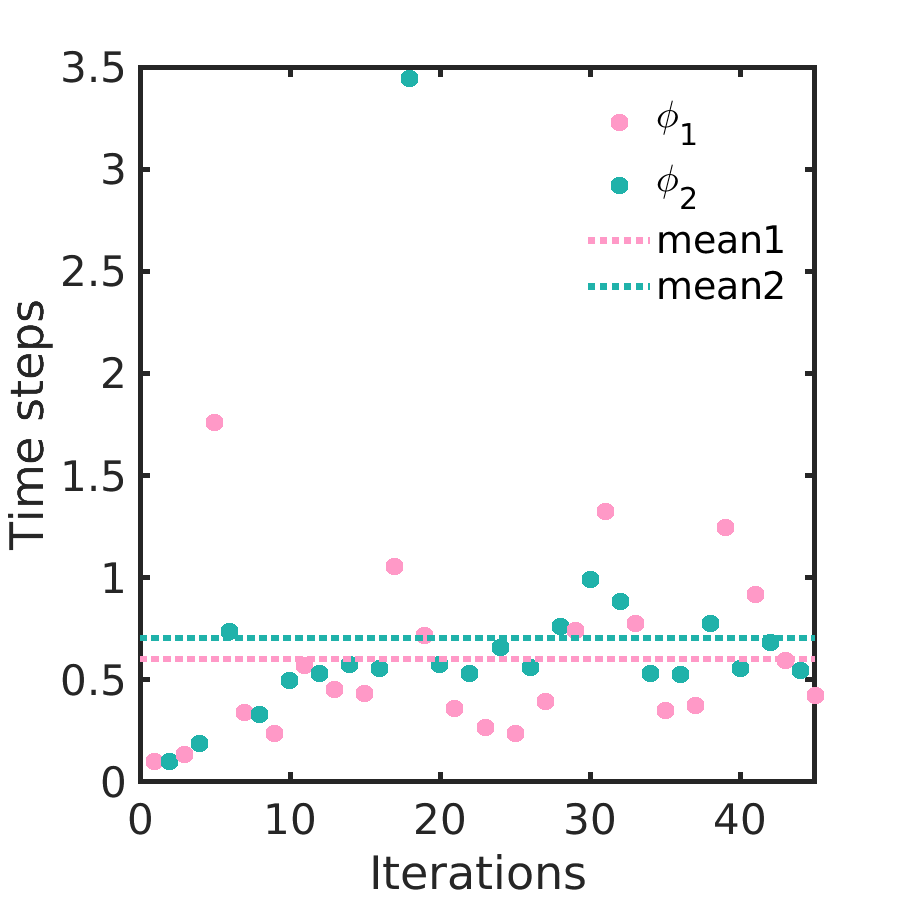}}
	\subfigure[]{\includegraphics[scale=0.25]{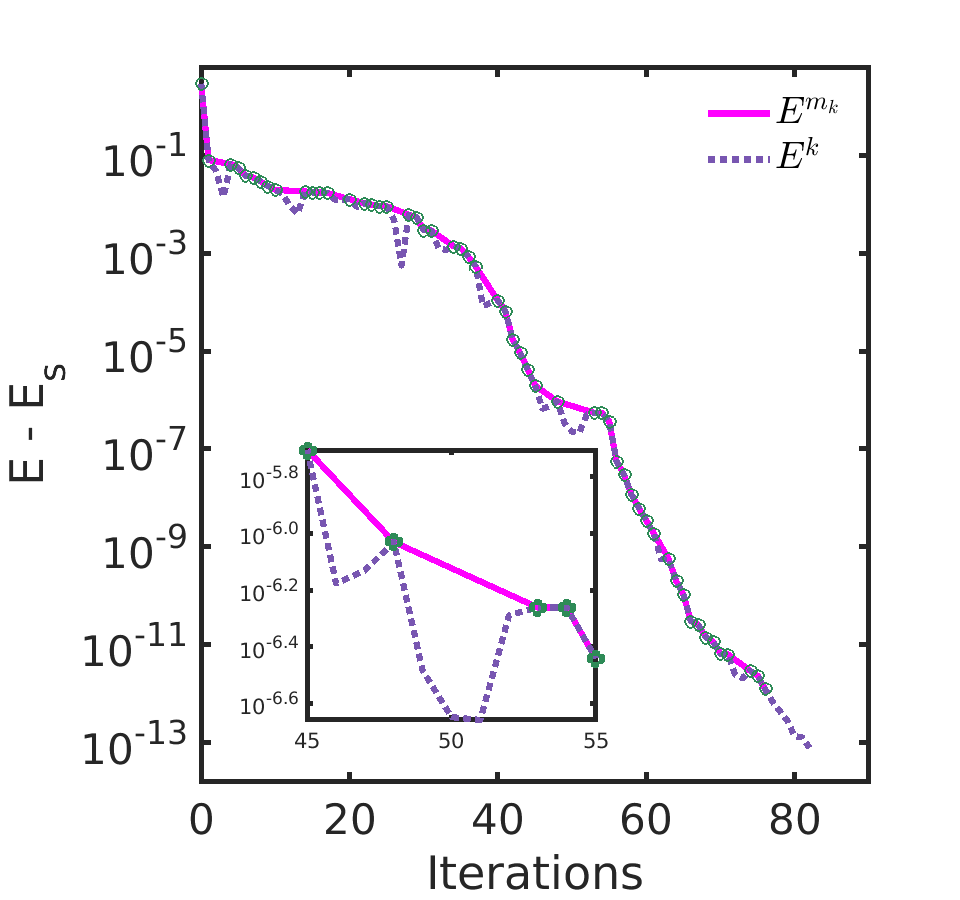}}
	\subfigure[]{\includegraphics[scale=0.25]{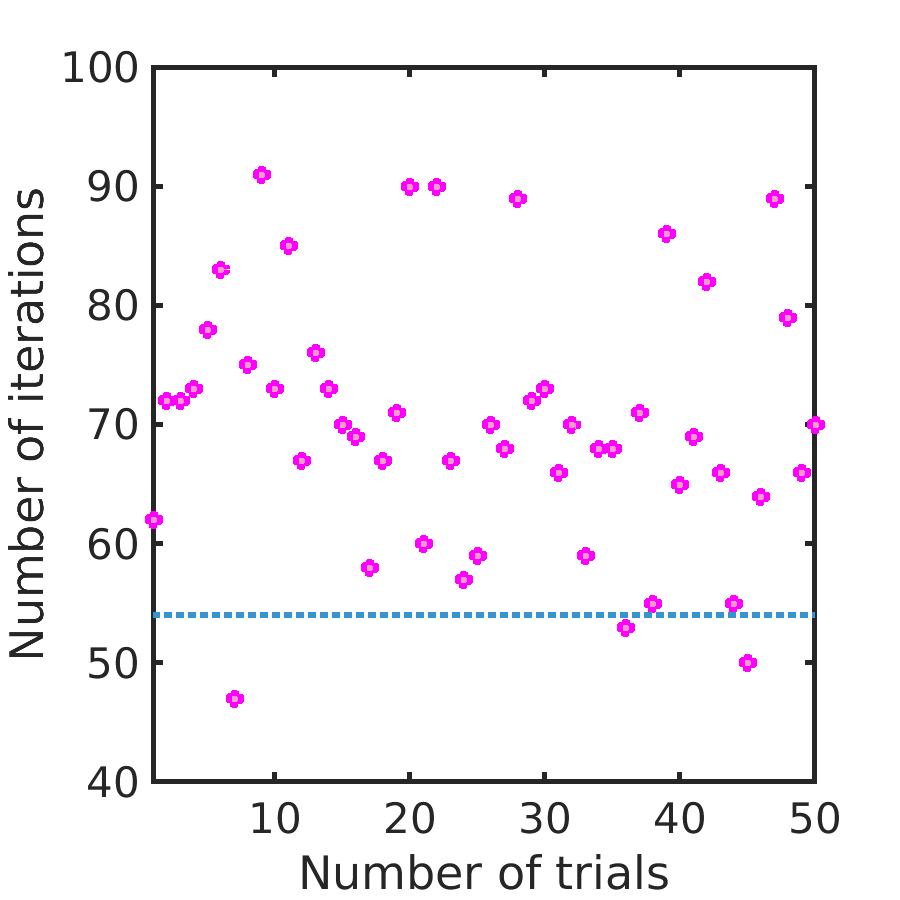}}
	\caption{
		\textbf{(a)}: The adaptive step size obtained by the AB-BPG-4 ($ M = 0 $, $ a = 1
		$). The mean step sizes of $ \phi_1 $ and $ \phi_2 $ are $0.6033$ and $0.7062$, respectively. 
		\textbf{(b)}: The tendency of energy $\{E^k=E(\hPhi^k)\}$ and energy $\{E^{m_k}=E(\hPhi^{m_k})\}$ in AB-BPG-2 ($ M = 5 $, $ a = 0
		$).
		The	green o's mark the position of $(m_k, E^{m_k})$.
		\textbf{(c)}: Numerical behavior of the AB-BPG-2 ($M=0$, $a=0$)  via updating order parameters in a random manner. The blue dotted line and the pink points denote the convergent iterations of the AB-BPG methods via cyclic and random order update, respectively.
	}
	\label{fig:AlgorithmStudy}
\end{figure}

\subsubsection{Comparison with other methods}
We compare AA-BPG methods with alternative methods, including adaptive SIS, adaptive BDF2, adaptive SAV and adaptive S-SAV. Theoretically, 
the SAV method always has a modified energy dissipation 
through adding a sufficiently large positive scalar auxiliary variable $C$
which guarantees the boundedness of the bulk energy term.
In practice, the original energy dissipation property might depend on the selection of $C$. 
When computing the decagonal quasicrystal in the binary CMSH model,
the adaptive SAV scheme keep the original energy dissipate when $C=10^8$. 
The times stepping for adaptive SAV scheme is obtained by formula \eqref{eq:adpstep2} with 
$\alpha_{\min} = 10^{-5}$, $\alpha_{\max} = 0.5$, $\rho = 0.9$, $\mathrm{tol} = 10^{-3}$.
For the AB-BPG approaches, we choose different values of $ M $ in  \eqref{def:m_k}
and  $ a $ in Bregman divergence \eqref{kernel} for comparison.
The linear search technique can obtain the step size of the AB-BPG methods.

Table \ref{tab:DQCres} shows the corresponding numerical results of the AB-BPG, adaptive SIS, adaptive BDF2,
adaptive SAV and adaptive S-SAV schemes.
We choose a reference energy value $ E_{s} = -1.54929536255898\times 10^{-2} $, which is computed via the semi-implicit scheme using $ 56^4 $ plane wave functions.
The minimal iteration steps and the least CPU time are emphasized via the bold font.  
From Table \ref{tab:DQCres}, one can find that the proposed approaches are superior to other methods. In particular, when
$a=1$, the AB-BPG method with $M=0$ takes $45$ iterations ($69.17$ seconds) to
reduce the gradient error of $10^{-7}$, which is $2.35$ times faster than the adaptive SIS,
$4$ times than the adaptive BDF2 and adaptive SAV schemes,
10.73 times than the adaptive S-SAV scheme. Moreover, a great deal of the adaptive time stepping for gradient flow methods is that empirical parameters are involved in the formula \eqref{eq:adpstep1} and \eqref{eq:adpstep2}, where the best-performing parameters usually cannot be guided in advance. Unsuitable choice of parameters may lead to too small step size (inefficient performance) or too large step size (divergence). While in our AB-BPG methods, efficient performances are shown in a wild range of choice of parameters. More importantly, the original energy dissipation and convergence are both guaranteed within any choice of parameters in our methods. Figure \ref{fig:DDQC1:comparison} presents the iteration process of
relative energy difference, CPU time, and gradient error of different approaches.

\begin{table}[htbp]
	\centering
	\caption{Numerical results of computing binary decagonal quasicrystal.}
	\begin{tabular}{|c|c|c|c|c| c|}
		\hline
		{$M$}                  & {$a$} & {Iterations}   & {CPU Time (s)}    & {Gradient error ($10^{-8}$)} & $|E-E_s|$ ($10^{-14}$)\\ \hline
		{}                     & {0}   & {54}           & {72.56}           & {9.74}  & 4.07\\ \cline{2-6}  
		                       & {0.1} & 46             & 66.04             & {9.42} & 51.54\\ \cline{2-6} 
		\multirow{-3}{*}{{0}}  & {1}   & {\textbf{45}}  & {\textbf{69.17}}  & {7.95} & 13.78\\ \hline
		{}                     & {0}   & {65}           & {86.62}           & {4.91} & 0.76\\ \cline{2-6} 
		                       & {0.1} & {59}           & {85.10}           & {9.25} & 23.08\\ \cline{2-6} 
		\multirow{-3}{*}{{5}}  & {1}   & {57}           & {82.79}           & {7.99} & 18.62\\ \hline
		{}                     & {0}   & {65}           & {85.88}           & {4.01}  & 0.06 \\ \cline{2-6} 
		{}                     & {0.1} & {59}           & {85.31}           & {9.60} & 95.61\\ \cline{2-6} 
		\multirow{-3}{*}{{10}} & {1}   & {56}           & {79.38}           & {7.13}  & 2.98\\ \hline
		\multicolumn{2}{|c|}{{Adaptive SIS}}            & {293}             & {163.10}    & {9.98}  & 152.28\\ \hline
		\multicolumn{2}{|c|}{{Adaptive BDF2}}           & {305}             & {276.98}    & {9.77} & 142.73\\ \hline
		\multicolumn{2}{|c|}{{Adaptive SAV}}            & {94}              & {283.14}    & {8.86} &   117.56 \\ \hline
		\multicolumn{2}{|c|}{{Adaptive S-SAV}}          & {236}             & {742.46}    & {9.81} &   141.65\\ \hline
	\end{tabular}
	\label{tab:DQCres}
\end{table}

\begin{figure}[htbp]
	\centering
	\centerline{\includegraphics[scale=0.09]{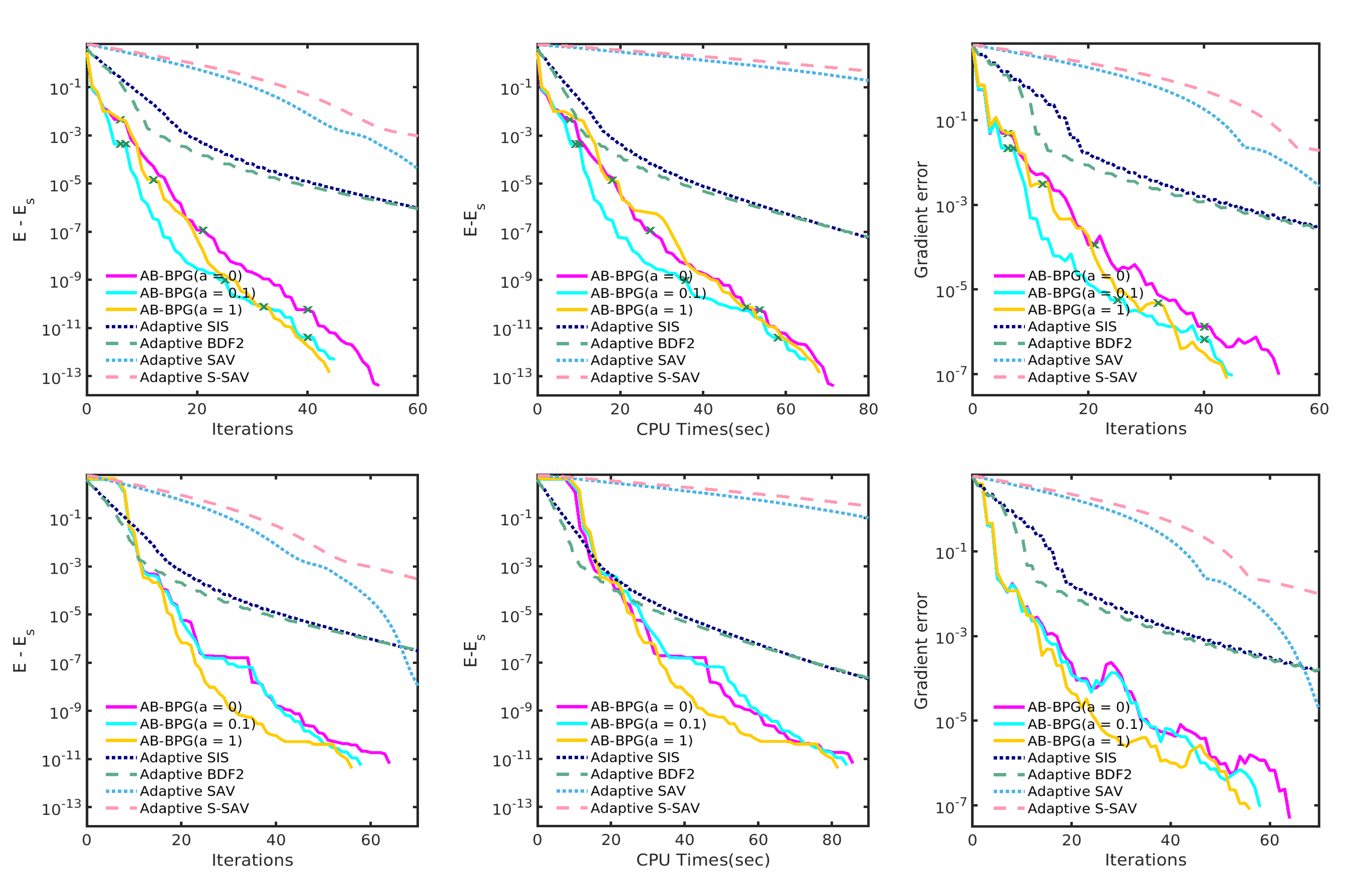}}
	
	
	\caption{Numerical behaviors of our algorithms, adaptive SIS, adaptive BDF2, adaptive SAV,
	adaptive S-SAV for computing binary
		decagonal quasicrystal. 
		\textbf{First row}: $ M = 0 $; \textbf{Second row}: $M = 5 $;
		\textbf{Left column}: Relative energy over iterations; \textbf{Middle column}:
		Relative energy over CPU time; \textbf{Right}: Gradient error over iterations;
		The green $ \times $s mark where restarts occurred. }
	\label{fig:DDQC1:comparison}
\end{figure}

\subsection{ Ternary component systems}

We consider the ternary component system when $s=3$ in the CMSH model. 
A periodic structure of the sigma phase, which is a complicated spherical packed structure discovered in multicomponent material systems\,\cite{lee2010discovery}, is used to examine the performance of our algorithms.
For such a pattern, we implement our algorithm 
on a bounded computational domain $[0, 27.7884)\times [0,
27.7884)\times [0, 14.1514)$ and  $200\times 200 \times 100$ plane wave 
functions are used to for discretization. 
The initial value of each component of the sigma phase is input, as suggested in\,\cite{xie2014sigma,arora2016broadly}.  
The parameters in the ternary CMSH model are given in Table \ref{tab:Sigmaparas}.
The stationary sigma phase is present in Figure \ref{fig:SigmaPhase}. 

\begin{table}[!htbp]
	\centering
	\caption{
		The non-zero model parameters used in computing ternary sigma phase.}
	\label{tab:Sigmaparas}
	\begin{tabular}{|c|}
		\hline
		\begin{tabular}[c]{@{}c@{}}$ c = 1 $, $ q_1 =q_2 = q_3=1 $, $ \tau_{2,0,0} = \tau_{0,2,0} = \tau_{0,0,2} = -0.2 $,
			$ \tau_{3,0,0} = \tau_{0,3,0} =  \tau_{0,0,3} = -0.3 $,\\ $ \tau_{4,0,0} =\tau_{0,4,0}=\tau_{0,0,4} = 0.1 $, 
			$ \tau_{2,1,1} = \tau_{1,2,1} =\tau_{1,1,2}= -0.1 $.
		\end{tabular} \\ \hline
	\end{tabular}
\end{table}

\begin{figure}[!htbp]
	\centering
	\centerline{\includegraphics[scale=0.15]{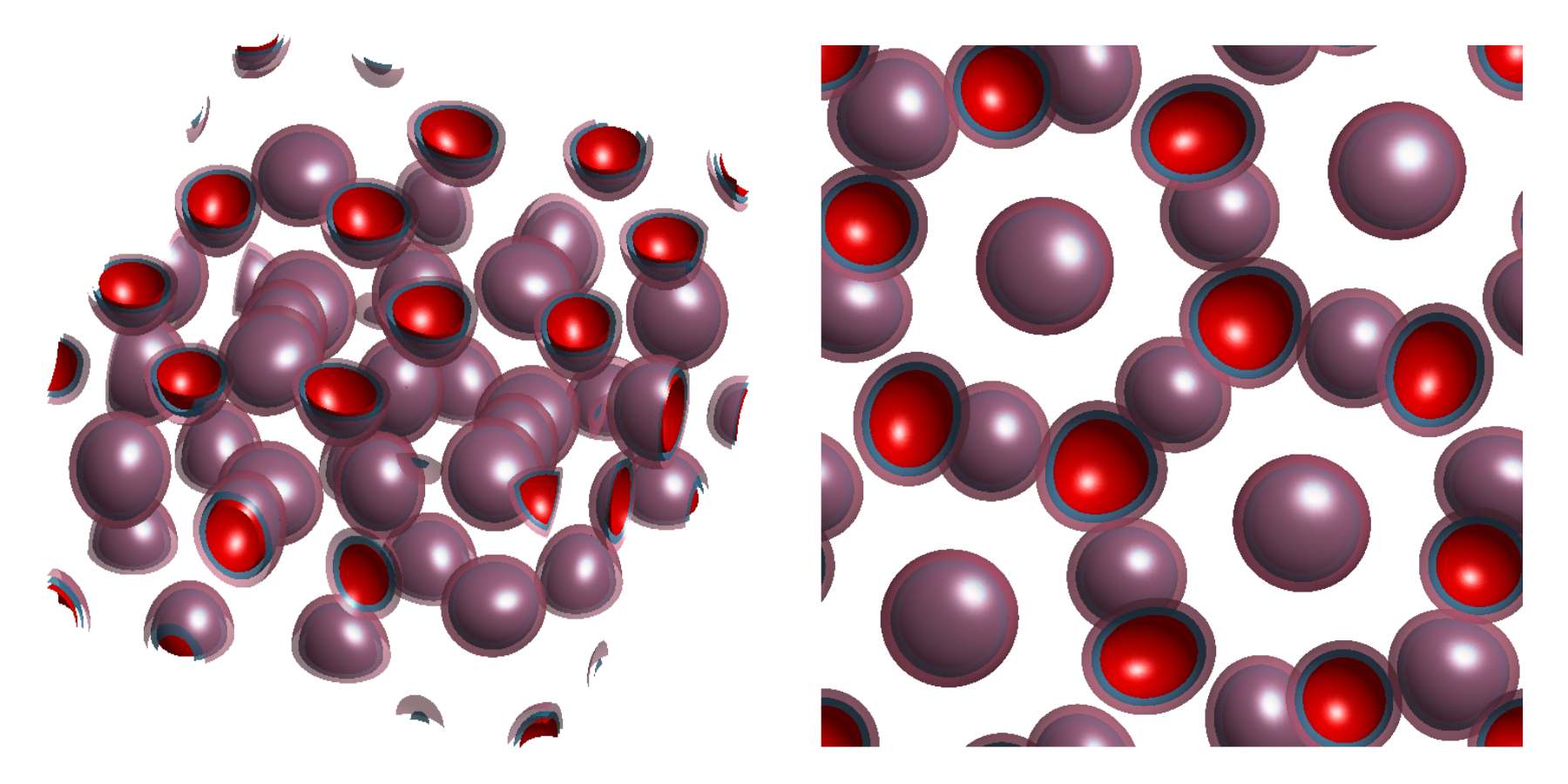}}
	\caption{The stationary sigma phase in ternary CMSH model from two perspectives. }
	\label{fig:SigmaPhase}
\end{figure}

Similar to the binary case, we report the AB-BPG algorithms' performance with different choices of $a$ and $M$.  In the adaptive SAV method, 
the auxiliary parameter $C$ is set to $ 10^{8} $ and the parameters in formula \eqref{eq:adpstep2} are chosen as
$\alpha_{\min} = 10^{-5}$, $\alpha_{\max} = 0.2$, $\rho = 0.8$, $\mathrm{tol} = 10^{-3}$.
Table \ref{tab:Sigmaresl} presents the corresponding numerical results and 
Figure \ref{fig:Sigma:comparison} gives the iteration process of different methods. 
The reference energy value $ E_s = -1.16910245253091 $ is obtained numerically
via semi-implicit scheme by using $256\times 256 \times 128$ plane wave functions.
As is evident from these results, the AB-BPG methods demonstrate a great advantage
in computing such a complicated periodic structure over other
schemes. More precisely, when $M=5$ and $a=1$, the 
AB-BPG method spends $416$ iterations to achieve the prescribed error, which
converges $16.5$ times faster than the adaptive SIS and adaptive BDF2 methods,
$ 3 $ times than the adaptive SAV method and $ 6.5 $ times than the adaptive S-SAV method. 
From the cost of CPU time, the AB-BPG algorithm with $M =5$, $a= 1$ takes $1358.93$ seconds to achieve
an accuracy of $10^{-7}$ in error, almost $10.7\%$, $10.8\%$, $11.9\%$  and $6.5\%$ the time
employed by the adaptive SIS, adaptive BDF2, adaptive SAV and adaptive S-SAV schemes.

\begin{table}[!htbp]
	\centering
	\caption{ Numerical results of computing ternary sigma phase.}
	\label{tab:Sigmaresl}
	\begin{tabular}{|c|c|c|c|c|c|}
		\hline
		{$M$}                  & {$a$} & {Iterations} & {CPU Time (s)} & Gradient error ($10^{-8}$)     & $|E-E_s|$  ($10^{-12}$)               \\ \hline
		                       & {0}   & {758}        & {2185.19}   & {6.47}  & 3.22\\ \cline{2-6} 
	                           & {0.1} & {770}        & {2549.49}   & {9.15}  & 3.57\\ \cline{2-6} 
		\multirow{-3}{*}{{0}}  & {1}   & {593}        & {1952.63}   & {9.71}  & 4.79 \\ \hline
		                       & {0}   & {739}           & {2068.65}          & {9.50}     & 3.54    \\ \cline{2-6} 
		                       & {0.1} & {611}        & {1991.94}   & {8.23} & 2.01\\ \cline{2-6} 
		\multirow{-3}{*}{{5}} & {1}   & {\textbf{416}}           & {\textbf{1358.93}}          & {8.51}          & 0.68          \\ \hline
		                       & {0}   & {677}           & {1896.60 }          & {7.13}    & 2.40             \\ \cline{2-6} 
    	                       & {0.1} & {595}        & {1959.80}   & {9.22} & 0.58\\ \cline{2-6} 
		\multirow{-3}{*}{{10}} & {1}   & {560}           & {1865.38}          & {7.32}    & 1.31            \\ \hline
		\multicolumn{2}{|c|}{{Adaptive SIS}}         & {6905}       & {12656.52}  & {9.96} & 6.28\\ \hline
		\multicolumn{2}{|c|}{{Adaptive BDF2}}        & {6896}       & {12549.01}   & {9.99} & 5.96\\ \hline
		\multicolumn{2}{|c|}{{Adaptive SAV}}        & {1276}       & {11406.12}   & {9.94}  & 6.11\\ \hline
		\multicolumn{2}{|c|}{{Adaptive S-SAV}}        & {2737}       & {20773.66}   & {9.99}  & 5.94\\ \hline
	\end{tabular}
\end{table}

\begin{figure}[!htbp]
	\centering
	\centerline{\includegraphics[scale=0.09]{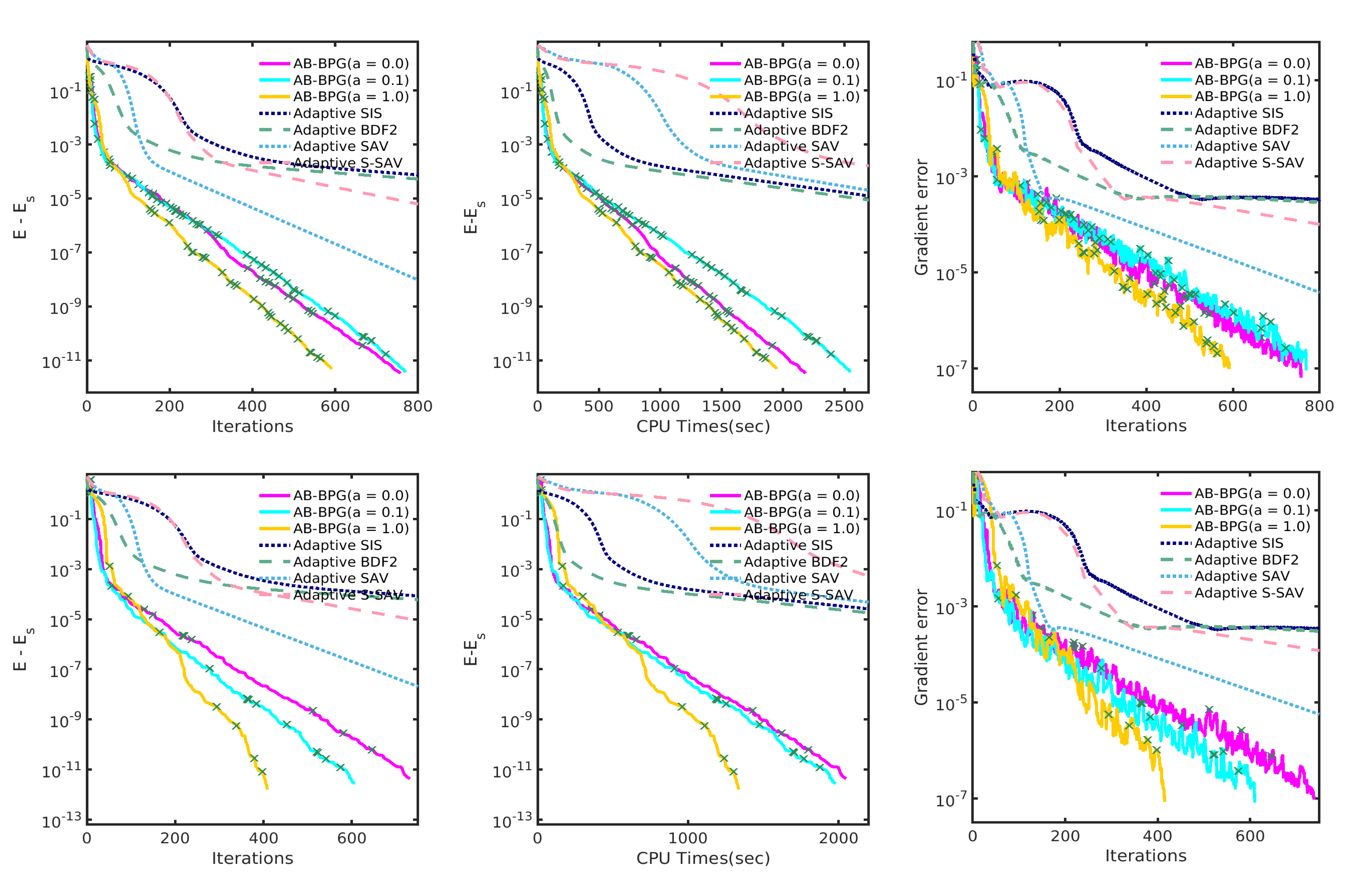}}
	%
	
	\caption{
		Numerical behaviors of the AB-BPG, adaptive SIS, adaptive BDF2, adaptive SAV and adaptive S-SAV algorithms for computing ternary sigma phase.
		\textbf{First row}: $ M = 0 $; \textbf{Second row}: $M = 20 $;
		\textbf{Left column}: Relative energy over iterations; \textbf{Middle column}:
		Relative energy over CPU time; \textbf{Right}: Gradient error over iterations;
		The green $ \times $s mark where restarts occurred. }
	\label{fig:Sigma:comparison}
\end{figure}

\subsection{ Quinary component systems} 

The last multicomponent system considered in this paper is the five component
CMSH model. 
We take a two-dimensional chessboard-shaped tiling phase and three-dimensional body-centered cubic (BCC) spherical structure to examine AB-BPG approaches' performance.  

\subsubsection{Chessboard-shaped tiling}
When computing the chessboard-shaped tiling, the parameters in five-component
CMSH model are given in Table \ref{tab:chessboardparas}. The corresponding computational
domain in physical space is $\Omega = [0,2\pi)^2$
and $1024\times 1024$ plane wave functions are used to discretize the
computational domain. The initial solution of $ j $-th component  is 
\begin{align}
\phi_j(\br) = \sum_{\bh\in\Lambda_0^j} \hphi(\bh)
e^{i \bh^\top \br},
~~\br\in\Omega,
\end{align}
where initial lattice points set $\Lambda_0^j\subset\bbZ^2$ can be found in the
Table \ref{tab:chessboardinitial} only on which the Fourier coefficients $\hphi(\bh)$
located are nonzero.  The convergent stationary morphology is given in
Figure \ref{fig:chessboard}.

\begin{table}[!htbp]
	\centering
	\caption{The non-zero model parameters when computing quinary chessboard-shaped tiling.}
	\label{tab:chessboardparas}
	\begin{tabular}{|c|}
		\hline
		\begin{tabular}[c]{@{}c@{}}$ c = 10 $, $ q_1 = q_2 = \cdots = q_5 =1$, \\$ \tau_{3,0,0,0,0} = \tau_{0,3,0,0,0} = \tau_{0,0,3,0,0} =\tau_{0,0,0,3,0} =\tau_{0,0,0,0,3} =-0.10 $,\\
			$ \tau_{4,0,0,0,0} =\tau_{0,4,0,0,0} = \tau_{0,0,4,0,0} = \tau_{0,0,0,4,0} = \tau_{0,0,0,0,4}  = 0.10 $,\\
			$ \tau_{1,0,1,0,0} = -0.70$, $ \tau_{0,1,0,1,1} = 0.05 $,  $ \tau_{1,1,0,0,1} = -0.12$, $ \tau_{0,1,0,1,0} = -0.44$.
		\end{tabular} \\ \hline
	\end{tabular}
\end{table}

\begin{table}[!htbp]
	\centering
	\caption{ The initial lattice points 
		of each component when computing quinary chessboard-shaped tiling.}
	\begin{tabular}{|c|c|c|c|c|c|}
		\hline
		{} & {$\phi_1$}    & {$\phi_2$}     & {$\phi_3$}    & {$\phi_4$}     & {$\phi_5$} \\ \hline
		{$ \Lambda_0^j $}            & {$(\pm 1,0)$} & {$(0, \pm 1)$} & {$(\pm 2,0)$} & {$(0, \pm 2)$} & {$(0, 0)$} \\ \hline
	\end{tabular}
	\label{tab:chessboardinitial}
\end{table}

\begin{figure}[!htbp]
	\centering
	\centerline{\includegraphics[scale=0.15]{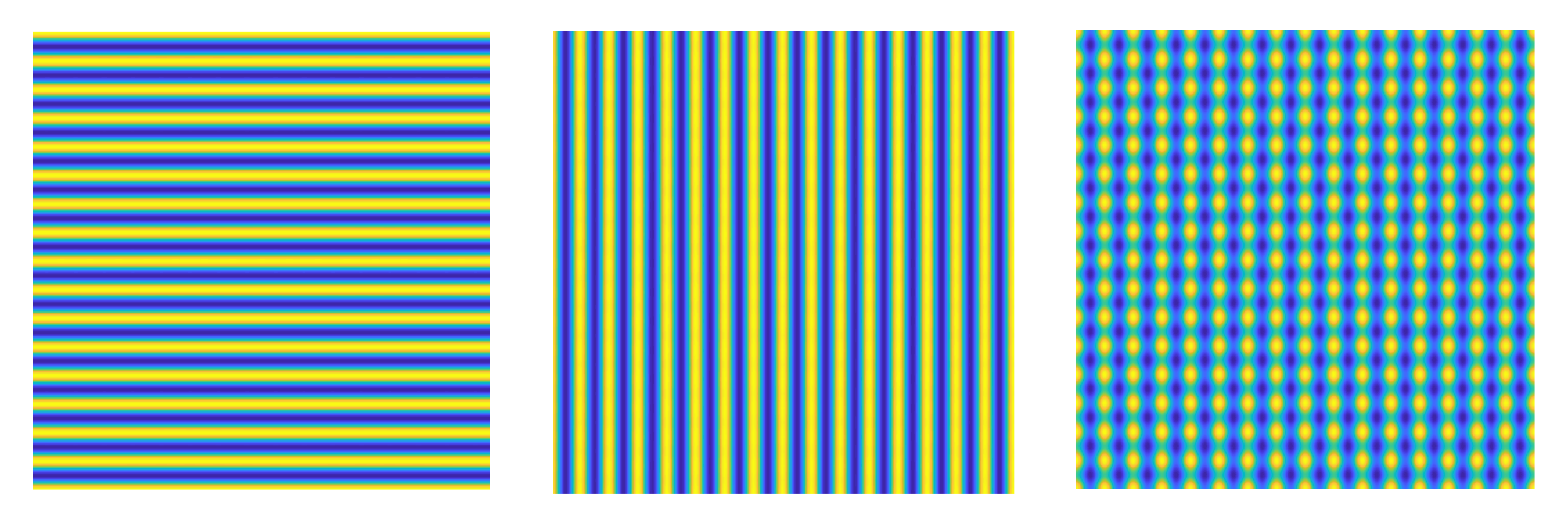}}
	\caption{The stationary  chessboard-shaped phase in quinary CMSH model. 
		The morphologies of $ \phi_1,\phi_3$ (\textbf{Left}), $\phi_2,\phi_4$
		(\textbf{Middle}) and $ \phi_5$ (\textbf{Right}). }
	\label{fig:chessboard}
\end{figure}

Table \ref{tab:chessboardresl} presents the numerical results of the AB-BPG, adaptive SIS, adaptive BDF2, adaptive SAV and adaptive S-SAV methods. 
The scalar auxiliary parameter $C$ of adaptive SAV scheme is set to $
10^{10} $ and the parameters in formula \eqref{eq:adpstep2} 
are taken as $\alpha_{\min} = 10^{-5}$, $\alpha_{\max} = 0.7$, $\rho = 0.9$, $\mathrm{tol} = 10^{-3}$.
The reference energy value $E_s=-0.57163687783216 $ is obtained 
via semi-implicit scheme by using $2048\times 2048$ plane wave functions.
Correspondingly, Figure \ref{fig:chessboard:comparison} presents the iteration process
including the relative energy difference, the CPU times and the gradient error against iterations, respectively.  
These results demonstrate the superiority of the AB-BPG algorithms over the adaptive SIS, adaptive BDF2 and adaptive SAV methods. As Table \ref{tab:chessboardresl} shows, the AB-BPG
method with $M=0$, $a=0$ has the best performance.
Even though the AB-BPG method costs
much time as the adaptive SIS and adaptive BDF2 approaches per iteration due to the linear search technique, its adaptive step size compensates the extra work converging to $10^{-7}$, almost $0.8\%$, $0.9\%$, 
$3.5\%$ and $0.4\%$ of CPU time spent by the adaptive SIS, adaptive BDF2, adaptive SAV methods, 
and adaptive S-SAV, respectively.

\begin{table}[!htbp]
	\centering
	\caption{Numerical results of computing the quinary chessboard-shape tiling}
	\label{tab:chessboardresl}
	\begin{tabular}{|c|c|c|c|c|c|}
		\hline
		$M$                 & $a$& Iterations& CPU Time (s)& Gradient error  ($10^{-8}$)        & $|E-E_s|$ ($10^{-14}$)    \\ \hline
		                    & 0  & \textbf{111}       & \textbf{75.56}   & 6.97 & 2.90\\ \cline{2-6} 
		                    & 0.1& {142}       & 101.52  &5.44 & 1.48\\ \cline{2-6} 
		\multirow{-3}{*}{0} & 1& {204}      & {145.30} & 6.60 & 1.19\\ \hline 
		                    & 0  & 193         & 149.67         & 6.28	         & 0.46              \\ \cline{2-6} 
		                    & 0.1& 174     & 131.11  &9.65 & 1.38\\ \cline{2-6} 
		\multirow{-3}{*}{5} & 1& 239     &171.46  & 9.91    & 1.72\\ \hline 
		                    & 0   & 186         & 120.03   &8.34    & 2.62\\ \cline{2-6} 
		                     & 0.1   & 188         & 135.34   &6.56 & 0.84\\ \cline{2-6} 
		\multirow{-3}{*}{10}& 1 & 253       & 186.38   & 5.37   & 0.52\\ \hline 
		\multicolumn{2}{|c|}{{Adaptive SIS} }         & 18208      & 8957.85  & 282.01    & 4617.51\\ \hline
		\multicolumn{2}{|c|}{{Adaptive BDF2}}         & 17723      & 8277.81 & 346.25 & 7073.17\\ \hline
		\multicolumn{2}{|c|}{{Adaptive SAV}}         & 927      & 2184.01  &  9.92  & 94.23\\ \hline
		\multicolumn{2}{|c|}{{Adaptive S-SAV}}         & 6655      & 16131.94  &  9.99  & 94.51\\ \hline
	\end{tabular}
\end{table}

\begin{figure}[!htbp]
	\centering
	\centerline{\includegraphics[scale=0.09]{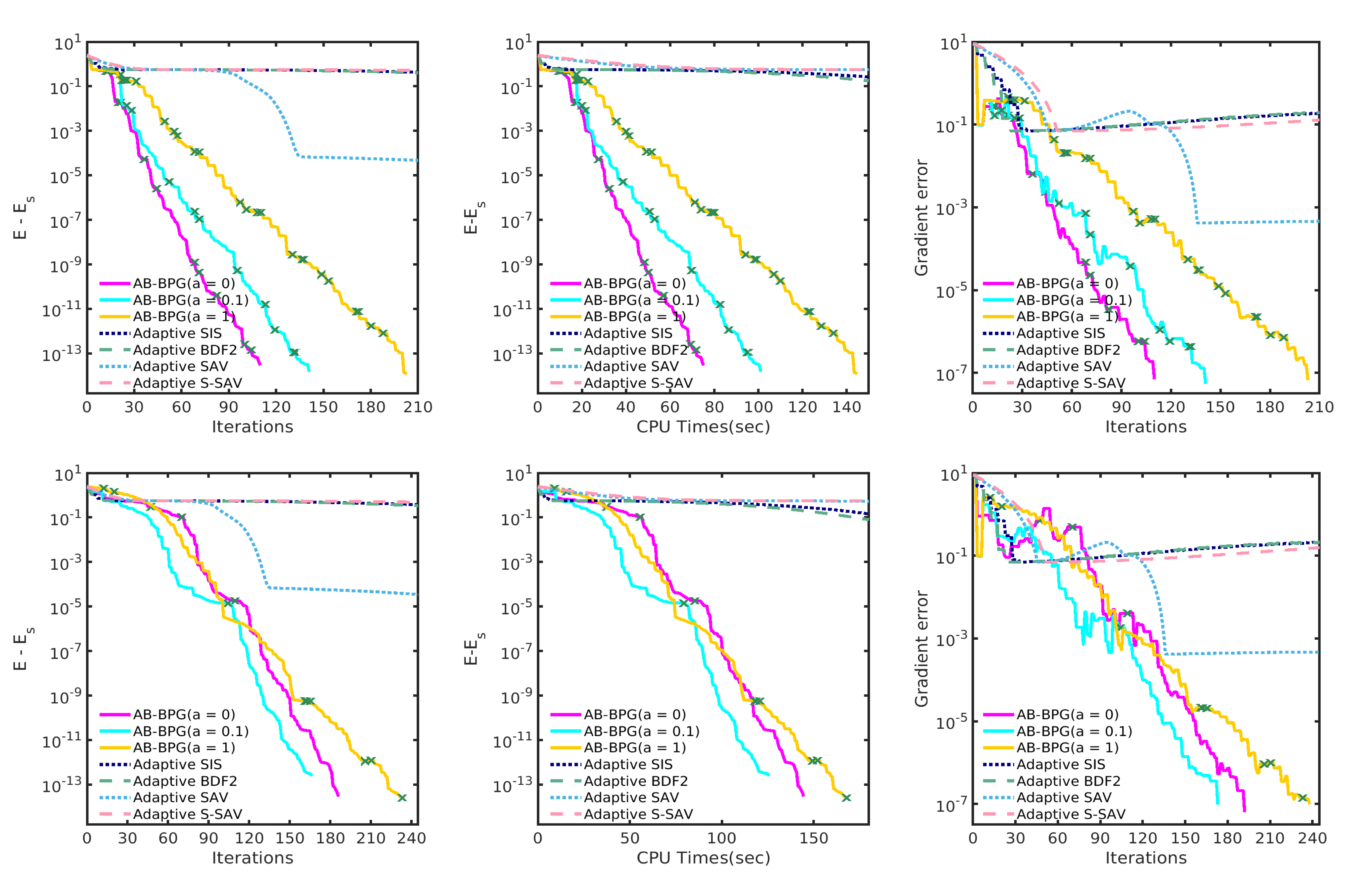}}
	%
	
	\caption{Numerical behaviors of the AB-BPG and adaptive SIS, adaptive BDF2, adaptive SAV and adaptive S-SAV for computing the quinary chessboard-shape tiling. \textbf{First row}: $ M = 0 $; \textbf{Second row}: $ M = 5 $;
		\textbf{Left}: Relative energy over iterations; \textbf{Middle}:
		Relative energy over CPU times; \textbf{Right}: Gradient error over iterations;
		The green $ \times $s mark where restarts occurred. }
	\label{fig:chessboard:comparison}
\end{figure}

\subsubsection{BCC}
As last, we consider the quinary BCC structure. The parameters in the five
component CMSH model are given in Table \ref{tab:BCCparas}. A bounded domain
$[0,2\sqrt{2}\pi)^3$ is used as the computational box and $128^3$ plane wave
functions are employed to compute the BCC spherical phase. The initial
values can be found in\,\cite{jiang2013discovery}. Figure \ref{fig:BCCsPhase}
shows the convergent stationary solution of different order parameter, which are
all BCC phases but with different periodicity.

\begin{table}[!htbp]
	\centering
	\caption{The non-zero model parameters used in computing quinary BCC spherical structure.}
	\label{tab:BCCparas}
	\begin{tabular}{|c|}
		\hline
		\begin{tabular}[c]{@{}c@{}}$ c = 1 $, $ q_1=1$, $ q_2 =1.5$, $ q_3 = 2$, $q_4 = 2.5$, $q_5 =3$, \\$ \tau_{3,0,0,0,0} = -0.1$, $ \tau_{0,3,0,0,0} = -0.6 $, $ \tau_{0,0,3,0,0} = -0.4 $, $\tau_{0,0,0,3,0}=-0.2$, $\tau_{0,0,0,0,3} =-0.1 $,\\
			$ \tau_{4,0,0,0,0} = \tau_{0,0,4,0,0}=\tau_{0,4,0,0,0} =\tau_{0,0,0,4,0} =\tau_{0,0,0,0,4}  = 0.1$,\\
			$ \tau_{1,0,1,0,0} = 0.4$, $ \tau_{0,1,0,1,0} = 0.3$, $ \tau_{0,1,1,1,0} = -0.2 $,  $ \tau_{1,1,0,0,1} = 0.8$.
		\end{tabular} \\ \hline
	\end{tabular}
\end{table}

Table \ref{tab:BCCresel} and Figure \ref{fig:BCC:comparison} compare the numerical behaviors of
AB-BPG, adaptive SIS, adaptive BDF2, adaptive SAV and adaptive S-SAV methods. The scalar auxiliary parameter
$C$ of two SAV scheme is set to $ 10^{10} $. For adaptive SAV sheme, the parameters in formula \eqref{eq:adpstep2} 
are taken as $\alpha_{\min} = 10^{-5}$, $\alpha_{\max} = 0.2$, $\rho = 0.9$, $\mathrm{tol} = 10^{-3}$. 
The reference energy $ E_s = -1.22314417498279 $ is obtained via semi-implicit scheme by using
$256^3 $ plane wave functions.
Again, for this case, the proposed AB-BPG methods are still superior to the compared algorithms. More precisely,  the best performance of AB-BPG ($
M = 0 $, $ a = 0 $) spends $182$ iterations to achieve the prescribed error
which converges $40$ times faster than the adaptive SIS, $30$ times than the adaptive BDF2
method, $5.1$ times than the adaptive SAV scheme and $10.1$ times than the adaptive S-SAV scheme. 
From the cost of CPU times, the AB-BPG algorithm with $M=0$ takes $247.22$
seconds to achieve
an accuracy of $10^{-7}$ in the gradient error, almost $3.8\%$, $5.1\%$, $
7.0\%$ and $2.8\%$ of the CPU time used by the adaptive SIS, adaptive BDF2,
adaptive SAV and adaptive S-SAV schemes.

\begin{figure}[!htbp]
	\centering	
	\centerline{\includegraphics[scale=0.3]{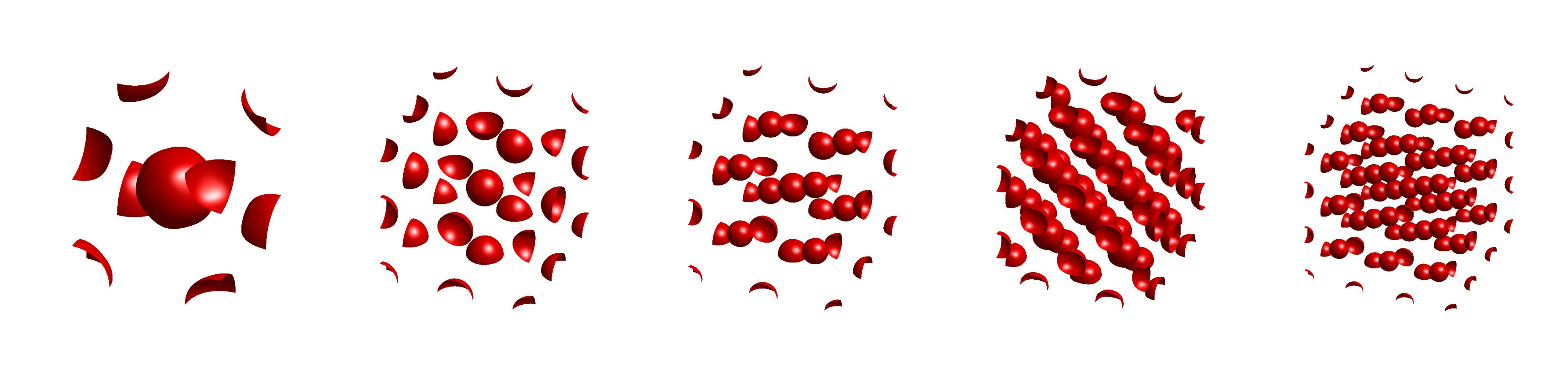}}
	\caption{The stationary BCC phase in quinary CMSH model. The subfigures from left to right are morphology of $ \phi_1$  to $ \phi_5 $.}
	\label{fig:BCCsPhase}
\end{figure}

\begin{table}[!htbp]
	\centering
	\caption{Numerical results of computing the quinary BCC structure}
	\label{tab:BCCresel}
	\begin{tabular}{|c|c|c|c|c|c|}
		\hline
		{$M$}                  & {$a$} & {Iterations} & {CPU Times} & {Gradient error ($10^{-8}$)}          & $|E-E_s|$ ($10^{-13}$)                \\ \hline
		{}                     & {0}   & {\textbf{182}}        & {\textbf{247.22}}    & {8.17}  & 1.99\\ \cline{2-6} 
		{}                     & 0.1   & {212}                 & 309.12               & 9.58    & 1.71      \\ \cline{2-6} 
		\multirow{-3}{*}{{0}}  & {1}   & {232}                 & {344.20}             & {6.66}  & 1.38\\\hline
		{}                     & {0}   & {259}                 & {343.63}             & {3.96}  & 1.37 \\ \cline{2-6} 
		{}                     & 0.1   & 262                   & 381.79               & 5.70    & 2.14\\ \cline{2-6} 
		\multirow{-3}{*}{{5}}  & {1}   & {347}                 & {515.89}             & {5.84} & 1.28\\ \hline
    		{}                 & {0}   & {277}                 & {371.94}             & {5.56} & 1.33\\ \cline{2-6} 
		{}                     & 0.1   & 304                  & 446.41               & 8.04 & 1.40  \\ \cline{2-6} 
		\multirow{-3}{*}{{10}} & {1}   & 353                  & 505.98               & 7.38   & 1.33\\ \hline
		\multicolumn{2}{|c|}{{Adaptive SIS}}         & {7443}        & {6446.39}    & {52.18} & 439.44\\ \hline
		\multicolumn{2}{|c|}{{Adaptive BDF2}}       & {5628}       & {4802.81}   & {71.90}  & 832.20\\ \hline
		\multicolumn{2}{|c|}{{Adaptive SAV}}       & {933}       & {3498.08}   & {9.94} & 17.20 \\ \hline
		\multicolumn{2}{|c|}{{Adaptive S-SAV}}       & {2204}       & {8766.82}   & {9.95} & 17.19 \\ \hline
	\end{tabular}
\end{table}
\begin{figure}[!htbp]
	\centering
	\centerline{\includegraphics[scale=0.09]{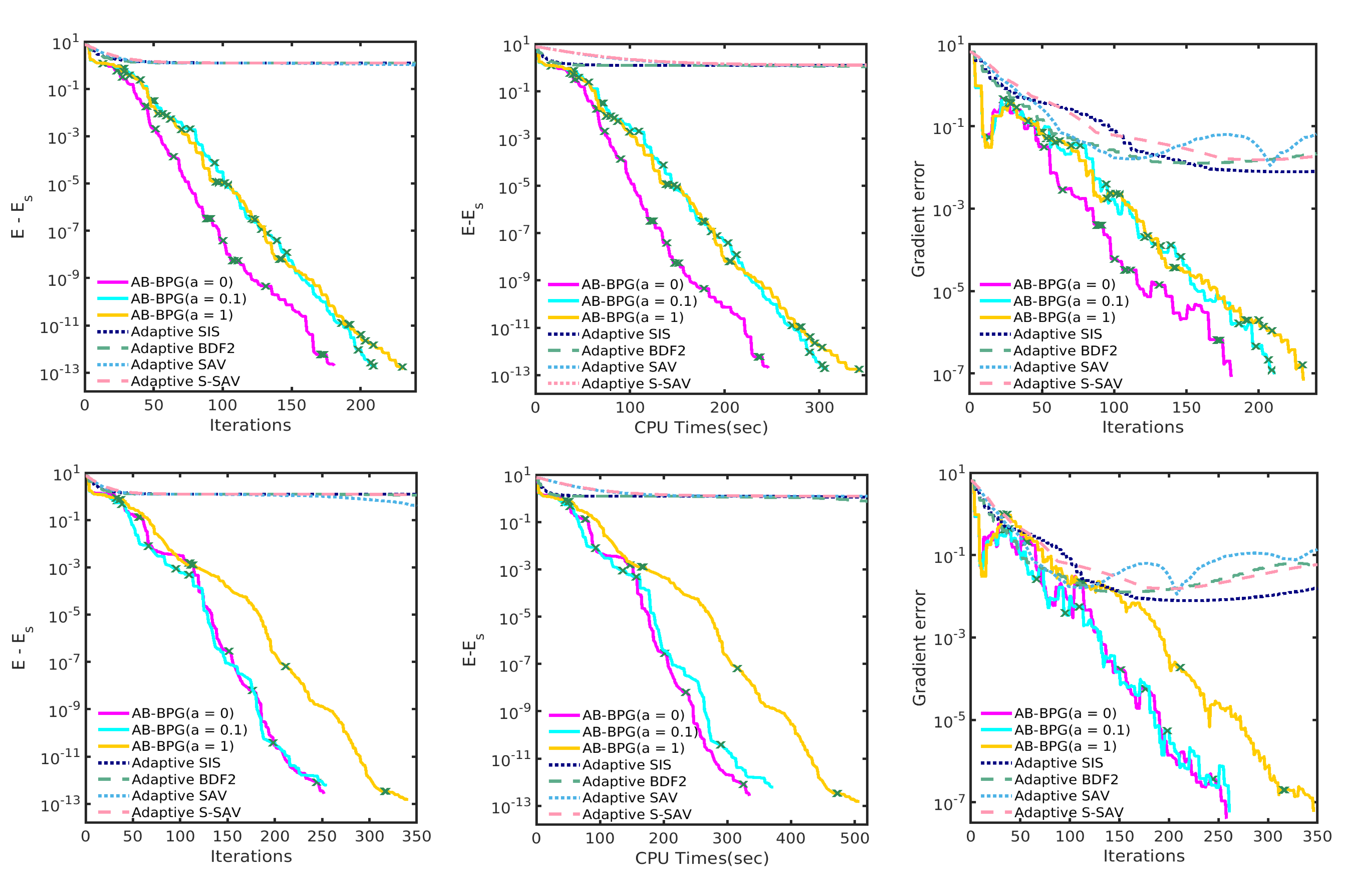}}
	%
	
	\caption{Numerical behaviors of the AB-BPG, adaptive SIS and adaptive BDF2, adaptive SAV and adaptive S-SAV methods for
		computing the quinary BCC phase. \textbf{First row}: $ M = 0 $; \textbf{Second row}: $ M = 5 $;
		\textbf{Left}: Relative energy over iterations; \textbf{Middle}:
		Relative energy over the CPU times; \textbf{Right}: Gradient error over iterations;
		The green $ \times $s mark where restarts occurred. }
	\label{fig:BCC:comparison}
\end{figure}


\section{Conclusion}
In this paper, an AB-BPG algorithm is proposed to compute the
stationary states of multicomponent phase-field crystal model with mass conservation.
Compared to most existing methods, the new approaches consider the block structure of multicomponent models, and the update manner can be chosen deterministically or randomly.
Using modern optimization methods, including the inertia acceleration approach, the restart technique, and the line search method, the proposed AB-BPG method has the general dissipation property with efficient implementation.
Moreover, with the help of the Bregman divergence, it is proved that the generated sequence converges to a stationary point without the requirement of the global Lipschitz assumption.
Extensive numerical experiments on computing stationary periodic crystals and
quasicrystals in the binary, ternary, and quinary coupled-mode Swift-Hohenberg model have shown a significant acceleration over many existing methods.

In this paper, we have implemented our algorithms for computing multicomponent model with polynomial type potentials $F[\{\phi_{j}\}_{j=1}^{s}] $. 
In fact, the proposed methods can be applied to deal with non-polynomial type potentials by choosing suitable $\{h_j\}$ to make our Assumption 
\ref{assum2} satisfied. This will soon be found in our future work. Besides, although our methods have shown efficient
performance in our numerical experiments, we want to theoretically prove the convergence rates in the future. It is worth mentioning that in this work, we only consider efficient methods for computing stationary states.  Besides, the physical evolution simulation is an important topic in PFC  models.   It is known that the classical proximal gradient method is exactly the semi-implicit scheme for gradient flows. From this perspective, we will develop our method to simulate the whole evolution of process of crystal growth in the further.

\section*{Acknowledgments}
This work is supported by the National Natural Science Foundation of China (11771368, 11901338). 
CLB is partially supported by Tsinghua University Initiative Scientific Research Program.
KJ is partially supported by the Key Project (19A500) of the Education Department of 
Hunan Province of China and the Innovation Foundation of Qian Xuesen Laboratory of Space Technology.

\bibliographystyle{plain}
\bibliography{references.bib}

\section*{Appendix A: Projection method}
	The projection method is a general framework to
	study the periodic and quasiperiodic crystals\,\cite{jiang2014numerical, Jiang2018numerical}. 
	Each of the $d$-dimensional periodic system can be described by a Bravais lattice
	\begin{align*}
	\mathcal{R}_d = \{\bA_d \bn_d, ~\bn_d\in\bbZ^d\},
	\end{align*}
	where $\bA_d\in\bbR^{d\times d}$ is invertible. 
	The fundamental domain or so-called unit cell of the periodic system is 
	\begin{align*}
	\Omega = \{\bA_d \bm{\zeta}, ~ \bm{\zeta}\in\bbR^d,~ \zeta_j \in [0,1),~
	j=1,\dots,d  \}.
	\end{align*}
	The associated primitive reciprocal vectors, $\bB_d=(\bb_1, \dots, \bb_d)$,
	$\bb_j \in \bbR^d$, $j=1,\dots,d$, satisfy the dual relationship 
	$\bA_d \bB_d^T = \bI_d$ where $\bI_d$ is the $d$-order identify matrix.
	The reciprocal lattice is 
	\begin{align*}
	\mathcal{R}^*_d = \{\bB_d \bh_d, ~ \bh_d\in\bbZ^d\}.
	\end{align*}
	The corresponding periodic function $\psi(\br)$, which has translational
	invariance with respect to the Bravais lattice $\mathcal{R}$, i.e.,
	$\psi(\br)=\psi(\br+\mathcal{R})$, has the following Fourier expansion 
	\begin{align}
	\psi(\br) = \sum_{\bh\in\bbZ^d} \hpsi(\bh) e^{i (\bB \bh)^T \br},
	\label{}
	\end{align}
	where 
	\begin{align}
	\hpsi(\bh) =\bbint \psi(\br)e^{-(\bB\bh)^T \br}d\br, \quad \br \in
	\Omega.
	\end{align}
	
	Quasiperiodic functions, or more general almost periodic functions, are an
	extension of periodic functions. The definition of quasiperiodic functions is given
	as follows\,\cite{Jiang2018numerical}:
	\begin{definition}[Quasiperiodic function]
		A $d$-dimensional continuous complex-valued function $f(\br)$ is
		quasiperiodic, if there exists a $n$-dimensional periodic function $F(\br_s)$,
		$\br_s\in \bbR^n$, $n\geqslant d$, such that 
		\begin{align*}
		f(\br) = F(\calP^T \br),
		\end{align*}
		where $\calP\in\bbR^{d\times n}$ is the projection matrix. The column vectors of
		$\calP$ are rationally independent. 
	\end{definition}
	From the definition, one can find that the $d$-dimensional quasiperiodic system is a
	$d$-dimensional subspace of a $n$-dimensional periodic structure. As describes above, the
	Fourier series of $F(\br_s)$ is 
	\begin{align}
	F(\br_s) = \sum_{\bh\in\bbZ^n} \hat{F}(\bh) e^{i (\bB \bh)^T \br_s},
	\label{}
	\end{align}
	where $\bB\in \bbR^{n\times n}$ is associated to the periodicity of the
	$n$-dimensional periodic system. 
	$\hat{F}(\bh) =\bbint F(\br_s)e^{-(\bB\bh)^T \br_s}\,d\br_s, \, \br_s \in
	\Omega_s $,
	$\Omega_s = \{\bB^{-T} \bm{\alpha}, ~ \bm{\alpha}\in\bbR^n, ~\alpha_j\in [0,1),~
	j=1,\ldots,n\}$. 
	Let $\hat{F}(\bh)$ be $\hat{f}(\bh)$, the projection method for a $d$-dimensional
	quasiperiodic function over $\bbR^d$ can be written as \cite{jiang2014numerical}
	\begin{align}
	f(\br) = \sum_{\bh\in\bbZ^n}\hat{f}(\bh)e^{i(\calP\bB\bh)^\top  \br},\quad
	\br \in\bbR^d.
	\label{}
	\end{align}
	If consider periodic crystals, the projection matrix becomes the $d$-order
	identity matrix, then the projection reduces to the common Fourier spectral method.
	
	Similarly, the order parameter in multicomponent systems can be expanded as follows
	\begin{align}\label{eq:pm}
	\phi_j(\br) =  \sum_{\bh\in\bbZ^n}\hphi_j(\bh)e^{i(\calP \bB\bh)^\top \br},\quad j=1,2,\cdots,s.
	\end{align}
	In numerical implementation, we truncate the Fourier coefficients to satisfy 
	\begin{equation*}
	X_j = \{\{\hphi_j(\bh)\}_{\bh\in\bbZ^n}:	\hphi_j(\bh) = 0, ~\forall\, |h_l|> \dfrac{N_{l,j}}{2}, ~ l = 1,2,\cdots,n\}.
	\end{equation*}
	where $N_{l,j}$ is chosen to be even for convenience.
	Let $ \bhphi_j = (\hphi_{1,j},\hphi_{2,j},\cdots,\hphi_{\bN_j,j})^\top\in\bbC^{N_j} $ with $ \bN_j  = \prod_{l=1}^n(N_{l,j}+1)$.
	Let $ \hPhi = \{\bhphi_j\}_{j =1}^s\in\bbC^{\bN} $ with $\bN = \sum_{j=1}^s\bN_j$. Using the projection method discretization, the energy functional \eqref{defined_GandF} reduce to
	\begin{align*}
	G_{\bh,j}(\bhphi_j) & = \frac{1}{2}\sum_{\bh_{j,1}  + \bh_{j,2} = 0}
	\left[q_j^2-(\mathcal{P}\mathbf{B}\bh)^\top (\mathcal{P}\mathbf{B}\bh)\right]^2
	\hphi_j(\bh_{j,1})\hphi_j(\bh_{j,2}),\quad j = 1,2,\dots,s,\\
	F_{\bh}(\hPhi) & = \sum_{\calI_{s,n}}\tau_{i_1,i_2,\cdots,i_s}\sum_{\sum_{j,k}\bh_{j,k} = 0}\prod_{j=1}^s\left(\prod_{k=1}^{i_j} \hphi_j(\bh_{j,k})\right),
	\label{eq:Energy_finite}
	\end{align*}
	where $\bh_{j, k}\in\bbZ^n$, $\hphi_j\in X_j$, $j=1,2,\dots,s$. 
	For simplicity, we omit the subscription $\bh$ in $G_{\bh,j}$ and $F_{\bh}$ in the following context.
	Then the discretized energy functional can be stated as
	\begin{equation}\label{Discret}
	\begin{split}
	E(\{\bhphi_j\}_{j=1}^s) = \sum_{j= 1}^{s}G_j(\bhphi_j) + F(\{\bhphi_j\}_{j=1}^s),
	\end{split}		
	\end{equation}
	where $ G_j(\bhphi_j) = \dfrac{1}{2}\langle \bhphi_j, \calD_j \bhphi_j\rangle $ and $ \calD_j\in\bbC^{\bN_j\times \bN_j} $ is a diagonal matrix with nonnegative entries
	\begin{equation}\label{def:Matrix_D}
	(\calD_j)_{\bh} =  [q_j^2-(\calP\bB\bh)^\top(\calP\bB\bh)]^2 \text{ with } \bh \in\bbZ^n \text{ and }\hphi_j(\bh)\in X_j.
	\end{equation}
	$ F(\{\bhphi_j\}_{j=1}^s) $ are $n$-dimensional convolutions in the reciprocal
	space. In summary, the discretized version of \eqref{infpro} has the form
	\begin{equation}
	\begin{split}
	\min_{\hPhi} ~  E(\hPhi) = \sum_{j= 1}^{s}G_j(\bhphi_j) + F(\hPhi),
	\quad\mathrm{s.t.}\quad  e_1^\top \bhphi_j = 0,\quad j = 1,2,\cdots,s.
	\end{split}
	\end{equation}

\section*{Appendix B: Proof of Theorem \ref{thm:SeqConvergence} }
Before we prove the convergent property, we first present a useful lemma for our analysis.
\begin{lemma}[Uniformized Kurdyka-Lojasiewicz property\,\cite{bolte2014proximal}.] \label{{lem:KL}}
	Let $\mathcal{C}$ be a compact set and
	$E$ defined in \eqref{finitepro} be bounded below. Assume that $E$ is constant on $\mathcal{C}$. Then, there exist $\epsilon>0$, $\eta>0$, and $\psi\in\Psi_\eta$ such that for all $\bar u\in\mathcal{C}$ and all $u\in\Gamma_\eta(\bar u,\epsilon)$, 
	one has,
	\begin{equation}\label{UKL}
	\psi^{'}(E(u)-E(\bar u))\dist(\vzero,\partial E(u))\geq 1,
	\end{equation}
	where $\Psi_\eta =\{\psi\in C[0,\eta)\cap C^1(0,\eta) \text{and } \psi \text{is concave}, \psi(0)=0, \psi^{'}>0 \text{on } (0,\eta)\}$ and $\Gamma_\eta(x,\epsilon) = \{y|\|x-y\|\leq \epsilon, E(x)<E(y)<E(x)+\eta\}$.
	\label{lemma:ukl}
\end{lemma}
\begin{proof}
	The proof is based on the fact that $ E $ satisfies the so-called Kurdyka-Lojasiewicz property on $\mathcal{C}$ \cite{bolte2014proximal}.
\end{proof}

\textbf{Proof of Theorem \ref{thm:SeqConvergence}}
\begin{proof}
	Define two sets $\Omega_2 = \{k\,|\,w_k = 0\}$ and $\Omega_1=\mathbb N\backslash\Omega_2$. Since $ M = 0 $, we know $ m_k\equiv k $. From the restart technique \eqref{restartcri}, the following sufficient decrease property holds	
	\begin{equation}\label{suffDec}
	E(X^k) - E(X^{k+1}) \geq \sigma\|X^k-X^{k+1}\|^2,~\forall k.
	\end{equation}
	Let $ S(X^0) $ be the set of limiting points of the sequence $ \{X^k\} $ starting
	from $ X^0 $. By the boundedness of $  \{X^k\} $ and $ S(X^k) =
	\cap_{n\in\mathbb{N}}\cup_{k\geq n}\{X^k\}$, it follows that $  \{X^k\} $ is a
	non-empty and compact set. From Lemma \ref{lem:lim_E}, we know 
	\begin{align}
	E(X) = E^*,\quad \forall X\in S(X^0).
	\end{align}			
	Then,  for all $  \varepsilon_1,\eta>0 $, there exists $ k_0>0 $ such that
	\begin{align}
	\dist(X^k, S(X^0))\leq \varepsilon_1,\quad E^*<E(X^k)<E^*+\eta, \quad  \forall k>k_0.
	\end{align}
	Applying Lemma \ref{lemma:ukl}, for $ \forall k>k_0 $ we have
	\begin{align}
	\psi'(E(X^k) - E^*)\dist(\vzero,\partial E(X^k))\geq 1.
	\end{align}
	By the convexity of $ \psi $, we have
	\begin{align}
	\psi(E(X^k) - E^*) - \psi(E(X^{k+1}) - E^*)\geq \psi'(E(X^k) - E^*)(E(X^k) - E(X^{k+1})).
	\end{align}
	Denote $ \Delta_{k,k+1} = \psi(E(X^k) - E^*) - \psi(E(X^{k+1}) - E^*) $. From \eqref{suffDec}, we have
	\begin{align}\label{ieq:KL2}
	\Delta_{k,k+1}\geq \dfrac{\sigma\|X^{k+1}-X^k\|^2}{\dist(\vzero,\partial E(X^k))},\quad \forall k> k_0
	\end{align}
	Next, we prove $ \sum_{k=0}^\infty\|X^{k+1}-X^k\| < \infty$. For any $ \varepsilon >0 $, we take $N >\max\{k_0,3T\} $ large enough such that 
	\begin{align}
	\bar{C}\psi(E(X^{N+1}) - E^*)<\varepsilon/2,\\
	\sum_{k=N-3T+2}^{N}\|X^{k+1}-X^{k}\| <\varepsilon/2,
	\end{align}
	where $ \bar{C} =3TC/\sigma  $. The existence of $ N $ is ensured by the fact that $
	\lim\limits_{k\rightarrow\infty}E(X^k) = E^* $, $ \psi(0) = 0 $ and $
	\lim\limits_{k\rightarrow\infty}\|X^{k+1}-X^k\| = 0 $. We now show that for all $
	\forall K>N $, the following inequality holds
	\begin{align}
	\sum_{k = N+1}^{K}\|X^{k+1}-X^k\|< \varepsilon,
	\end{align}
	which implies $ \sum_{k=0}^\infty\|X^{k+1}-X^k\| < \infty$ by Cauchy principle of
	convergence. In fact, together with Lemma \ref{lem:grad2} and \eqref{ieq:KL2}, we obtain
	\begin{align}
	\Delta_{k,k+1}\geq \dfrac{\|X^{k+1}-X^k\|^2}{C_1\sum_{l= k-3T+1}^k\|X^l-X^{l-1}\|},\quad \forall k>N
	\end{align}
	where $ C_1 = C/\sigma. $ By using the geometric inequality, we get  
	\begin{align}\label{sum}
	2\|X^{k+1}- X^k\|\leq \dfrac{1}{3T}\sum_{l= k-3T+1}^k\|X^l-X^{l-1}\| + \bar{C}\Delta_{k,k+1},
	\end{align}
	Summing up \eqref{sum} for $ k=N+1,\cdots,K $, it follows that
	\begin{align*}
	2\sum_{k = N+1}^K\|X^{k+1}- X^k\|&\leq \sum_{k=N+1}^K\left(\dfrac{1}{3T}\sum_{l= k-3T+1}^k\|X^l-X^{l-1}\| + \bar{C}\Delta_{k,k+1}\right)\\
	& =\dfrac{1}{3T}\sum_{k=N+1}^K\sum_{l= -3T+1}^0\|X^{l+k}-X^{l+k-1}\| + \bar{C}\Delta_{N+1,K+1}\\
	& =\dfrac{1}{3T}\sum_{l= -3T+1}^0\sum_{k=N+l+1}^{K+l}\|X^k-X^{k-1}\| + \bar{C}\Delta_{N+1,K+1}\\
	&\leq \dfrac{1}{3T}\sum_{l= -3T+1}^0\sum_{k=N-3T+2}^K\|X^{k}-X^{k-1}\| + \bar{C}\Delta_{N+1,K+1}\\
	&= \sum_{k=N-3T+2}^K\|X^{k}-X^{k-1}\| + \bar{C}\Delta_{N+1,K+1}\\
	&\leq \sum_{k=N-3T+1}^K\|X^{k+1}-X^{k}\| + \bar{C}\Delta_{N+1,K+1},
	\end{align*}
	where the first equality is from the fact that $ \Delta_{p,q} + \Delta_{q, r} = \Delta_{p,r} $. Therefore, we have
	\begin{align}
	\sum_{k = N+1}^K\|X^{k+1}- X^k\|\leq  \sum_{k=N-3T+1}^{N}\|X^{k+1}-X^{k}\| + \bar{C}\psi(E(X^{N+1}) - E^*)<\varepsilon
	\end{align}
	As a result, we obtain $ \sum_{k=0}^{\infty}\|X^{k+1}-X^k\|<\infty $ and $\lim\limits_{k\rightarrow\infty}X^k = X^*$. From 
	Lemma \ref{lem:lim_E} and the continuity of $ E(X)$ on $ \calB(X^0) $, we get
	\begin{align}
	\lim\limits_{k\rightarrow\infty} E(X^k) = E(X^*) = E^*.
	\end{align}
\end{proof}

\end{document}